\documentclass[11pt]{amsart}
\usepackage{mathrsfs}
\usepackage{amsfonts}
\usepackage[dvipdfm, margin=3cm]{geometry}
\usepackage{amsmath}
\usepackage{amssymb}
\usepackage{amsthm}
\usepackage{latexsym}
\allowdisplaybreaks[4]
\numberwithin{equation}{section}
\newtheorem{thm}{Theorem}[section]
\newtheorem{prop}[thm]{Proposition}
\newtheorem{cor}[thm]{Corollary}
\newtheorem{lem}[thm]{Lemma}

\newtheorem{rem}[thm]{Remark}

\DeclareMathOperator{\p}{p}

\DeclareMathOperator{\Hom}{Hom}

\def\cal{\mathcal}

\def \Z{\Bbb Z}
\def \C{\Bbb C}
\def \R{\Bbb R}

\def \Or{{\bf O}}

\def \tr{{\rm tr}}

\def \qdim{{\rm qdim}}
\def \d{{\rm d}}
\def \Hom{{\rm Hom}}

\def \KW {{\rm KW}}

\def \<{\langle}
\def \>{\rangle}

\def \l{\lambda }

\def \g{\gamma}

\def\g{\mathfrak g}
\def\p{\mathfrak p}
\def\CC{{\mathcal{C}}}
\def\FF{{\mathcal{F}}}
\def\DD{{\mathcal{D}}}
\def\1{{\bf 1}}
\def\id{{\rm id}}
\def\FPdim{{\rm FPdim}}

\newcommand{\BR}{\mathbb{R}}

\def \1{{\bf 1}}
\begin{document}
\title{\bf Coset constructions and Kac-Wakimoto Hypothesis}
\author{Chongying Dong}
\address{ Department of Mathematics, University of California, Santa Cruz, CA 95064}
 \email{dong@ucsc.edu}
 \thanks{The first  author is supported by the Simons foundation  634104}
\author{Li Ren}
\address{School of Mathematics, Sichuan University,
 Chengdu 610064 China }
 \email{renl@scu.edu.cn}
 \thanks{The second author is supported by China NSF grant 11301356}
 \author{Feng Xu}
 \address{ Department of Mathematics, University of California, Riverside, CA 92521}
 \email{xufeng@math.ucr.edu }
\maketitle

\begin{abstract}
Categorical coset constructions are investigated and Kac-Wakimoto Hypothesis associated with pseudo unitary modular tensor categories is proved. In particular, the field identifications are obtained. These results are applied to the coset constructions in the theory of vertex operator algebra.

\end{abstract}

\section{Introduction}
\setcounter{equation}{0}

Introduced in \cite{GKO}, the coset construction deals with a pair of rational conformal field theories, and is a very powerful tool in the study of conformal field theory and vertex operator algebras.
It has been conjectured \cite{MS}, \cite{W} that all rational conformal field theories are related to the coset constructions, orbifold constructions and extensions. While the orbifold theory is relatively better understood due to the work in \cite{FLM, DVVV, DHVW, DM, DLM1, DLM3, DLM4, HMT, DJX, DRX, DNR}, the coset theory has not been investigated much in general except in conformal nets setting \cite{X1,X2,X3,X4}. 

From the point of view of vertex operator algebra, the coset construction studies the commutant $U^c$ \cite{FZ} of a vertex operator subalgebra $U=(U,Y,{\bf 1},\omega^1)$ in vertex operator algebra $V=(V,Y,{\bf 1},\omega).$ Assume that $L(1)\omega^1=0$ where
$L(1)$ is the component operator of $Y(\omega,z)=\sum_{n\in\Z}L(n)z^{-n-2}.$
Then the commutant
$U^c=\{v\in V|u_nv=0 \ for \ u\in U, n\geq 0\}$ is also
vertex operator subalgebra of $V$ with Virasoro element
$\omega^2=\omega-\omega^1.$ The $U^c$ is called the coset vertex
operator algebra associated with the pair $U\subset V.$
It is clear from the definition that $Y(u,z)$ and $Y(v,z)$ for
$u\in U,v\in U^c$ commute on any weak $V$-module and this is how the duality of Schur-Weyl type appears in the coset construction. One of the main problems is how to decompose an arbitrary irreducible $V$-module into a direct sum of irreducible  $U\otimes U^c$-modules.

We approach the coset constructions categorically in this paper. Let $\CC_1, \CC_2$ be pseudo unitary modular tensor categories. We use $\Or(\CC_1)=\{W^\alpha|\alpha\in J\}$  to denote the isomorphism classes of simple objects of $\CC_1$ and we assume that  $W^1=1_{\CC_1}.$  Let $A\in \CC_1\otimes \CC_2$ be a regular commutative algebra. Then local $A$-modules $\CC=(\CC_1\otimes \CC_2)_A^0$ is also a modular tensor category. Let  $\Or(\CC)=\{M^i|i\in I\}$ be the isomorphism classes of simple objects of $\CC$  with  $M^1=A.$ 
Then 
$$M^i\cong\bigoplus_{\alpha\in J_i}W^\alpha\otimes M^{(i,\alpha)}$$ 
as objects in $\CC_1\otimes \CC_2$ for $i\in I$ and a subset $J_i$ of $J.$   The main goal is to understand $M^{(i,\alpha)}$ under the assumptions that $\CC_1,\CC_2$  are pseudo unitary,  $M^{(1,1)}=1_{\CC_2}$ and $\dim\Hom_{\CC_2}(1_{\CC_2},M^{(1,\alpha)})=\delta_{1_{\CC_1},\alpha}.$ 
It turns out that the Kac-Wakimoto set 
$$\KW=\{W^\beta\in \Or(\CC_1)|A\boxtimes_{\CC_1\otimes \CC_2}(W^\beta\otimes 1_{\CC_2})\in\CC\}$$
plays an essential role in studying $M^{(i,\alpha)}.$  Note that $1_{\CC_1}\in \KW.$ So the Kac-Wakimoto set is not empty. For short, set 
$$a_{\beta\otimes 1}=A\boxtimes_{\CC_1\otimes \CC_2}(W^\beta\otimes 1_{\CC2}).$$ 
Use the  Kac-Wakimoto set to define an equivalence relation on $\Or(\CC)$ such that $M^i,M^j$ are equivalent if and only if there exists $W^\beta\in \KW$ and $M^i$ is a direct summand of $M^j\boxtimes_A a_{\beta\otimes 1}.$ Then $M^i_{\CC_2}=M^j_{\CC_2}$ if $M^i,M^j$ are equivalent and  $M^i_{\CC_2}\cap M^j_{\CC_2}=\emptyset$ otherwise where $M^i_{\CC_2}$ is the simple objects of $\CC_2$ appearing in $M^i.$ This gives the first identification between $M^{(i,\alpha)}$ and $M^{(j,\gamma)}$ for $\alpha\in J_i$ and $\gamma\in J_j.$

Using the  Kac-Wakimoto set we can also give an upper bound 
of  $\dim\Hom_{\CC_2}(M^{(i,\alpha)}, M^{(i,\alpha)})$ and a precise and new formula $\dim M^{(i,\alpha)}=c_i d_id_\alpha$ where $c_i=\frac{\sum_{\beta\in J_1}d_{\beta}^2}{\sum_{\gamma\in J_i}d_{\gamma}^2}\leq 1,$ $d_i=\dim M^i$ and $d_\alpha=\dim W^\alpha.$ If $\KW$ consists of $1_{\CC_1}$ only then all $M^{(i,\alpha)}$ are inequivalent simple objects in  $\CC_2.$ Furthermore,  $c_i=1$ for all $i$ is equivalent to that $\KW$ forms a group. In this case, all the simple objects in $\CC_2$ appearing in $M^{(i,\alpha)}$ have the same dimension. If $\KW$ forms a cyclic group, all the simple objects in $\CC_2$ appearing in $M^{(i,\alpha)}$ are multiplicity-free.

The importance of the Kac-Wakimoto set was first noticed in \cite{KW} when they studied the branching functions
associated with a pair of simple Lie algebras $\p\subset \g$ \cite{K}. Let $\hat{\g}$ be the corresponding affine Kac-Moody algebra, $P_+^k$ be the set of dominant weights of level $k$ which is a positive integer, $L_{\hat\g}(\Lambda)$ is the corresponding irreducible highest weight module with highest weight $\Lambda\in P_+^k.$ Then $L_{\hat\g}(\Lambda)$ is a $
\hat\p$-module of level $\dot{k}$ with the following decomposition
$$L_{\hat\g}(\Lambda)=\bigoplus_{\lambda\in \dot{P}_+^{\dot{k}}}L_{\hat\p}(\lambda)\otimes L(\Lambda,\lambda)$$
where $L(\Lambda,\lambda)$ is the multiplicity space of $L_{\hat\p}(\lambda).$ Let $h_{\Lambda,\lambda}$ be the minimal weight of $L(\Lambda,\lambda).$ 
Let $$E=\{(\Lambda,\lambda)|\Lambda\in P_+^k, \lambda\in\dot{P}_+^{\dot{k}},L(\Lambda,\lambda)\ne 0\}.$$
Then Kac-Wakimoto set  in this case can be identified with 
$$\KW=\{(\Lambda,\lambda)\in E|h_{\Lambda,\lambda}=0\}.$$
Let $\chi_{\Lambda}(\tau)$ be  the character of $L_{\hat\g}(\Lambda)$ and  $\dot {\chi}_{\lambda}(\tau)$ be the character of $L_{\hat\p}(\lambda).$ It follows from \cite{KP} that the space spanned by the irreducible characters is a representation of $SL_2(\Z).$ In particular,
$$	\chi_{\Lambda}(-1/\tau)=\sum_{M\in P_+^k}s_{\Lambda M}\chi_{M}(\tau).$$ 
Similarly,
$$\dot\chi_{\lambda}(-1/\tau)=\sum_{\mu\in  \dot P_+^{\dot k}}\dot s_{\lambda \mu}\dot\chi_{\mu}(\tau).$$ 
To understand the asymptotic properties of branching functions - characters of $L(\Lambda,\lambda),$ Kac-Wakimoto proposed the following hypothesis: For any $(\Lambda,\lambda)\in E, (M,\mu)\in \KW,$
$$s_{\Lambda M}\overline{\dot s_{\lambda\mu}}\geq 0.$$
Unfortunately, there are counterexamples to the Kac-Wakimoto hypothesis \cite{X3, X4}. 

Another goal of this paper is to explain why the Kac-Wakimoto hypothesis does not hold in these examples. We also give a sufficient condition under which the Kac-Wakimoto hypothesis holds using the categorical coset construction setting.  Note that both $L_{\hat\g}(k\Lambda_0)$ and $L_{\hat\p}(\dot k\lambda_0)$ are rational, $C_2$-cofinite vertex operator algebras whose irreducible modules are 
$\{L_{\hat\g}(\Lambda)|\Lambda\in P_+^k\}$, and $\{L_{\hat\p}(\lambda)|\lambda\in \dot{P}_+^{\dot{k}}\},$
respectively, where $\Lambda_0, \lambda_0$ are the fundamental weights associated to the index $0.$ 
Then each $L(\Lambda,\lambda)$ is a module for the coset VOA $L_{\hat\p}(\dot k\lambda_0)^c.$ In general,
 $L_{\hat\p}(\dot k\lambda_0)^{cc}$ may not equal to   $L_{\hat\p}(\dot k\lambda_0).$ 
We denote the Virasoro operators of $L_{\hat\g}(k\Lambda_0)$ and $L_{\hat\p}(\dot k\lambda_0)$ by
 $L_\g(n),$ and $L_\p(n),$ respectively. Then each $L(\Lambda,\lambda)$ is a unitary representation of 
  the Virasoro operators $L_\g(n)-L_\p(n).$ This implies that  $h_{\Lambda,\lambda}\geq 0$ and 
  $h_{\Lambda,\lambda}=0$ if and only if $L_{\hat\p}(\dot k\lambda_0)^c$ is a submodule of $L(\Lambda,\lambda).$ So we can define $\KW$-set as
  $$\KW=\{(\Lambda,\lambda)\in E| \Hom_{L_{\hat\p}(\dot k\lambda_0)^c} (L_{\hat\p}(\dot k\lambda_0)^c,L(\Lambda,\lambda))\ne 0\}.$$
  Now let $\CC$ be the module category of $L_{\hat\g}(k\Lambda_0),$ $\CC_1$ the module category of $L_{\hat\p}(\dot k\lambda_0)$ and $\CC_2$ the module category of $L_{\hat\p}(\dot k\lambda_0)^c.$
  From \cite{H}, both $\CC$ and $\CC_1$ are modular tensor categories. Also according to the coset construction conjecture, $L_{\hat\p}(\dot k\lambda_0)^c$ should be rational and $C_2$-cofinite, and $\CC_2$ is also a modular tensor category. Moreover, $L_{\hat\g}(k\Lambda_0)$ is a regular commutative algebra in $\CC_1\otimes \CC_2.$ So we can study the Kac-Wakimoto hypothesis in a general categorical coset setting.  With the assumptions in a categorical setting, we show that the Kac-Wakimoto hypothesis holds. Notice that the assumptions that $M^{(1,1)}=1_{\CC_2}$ and $\dim\Hom_{\CC_2}(1_{\CC_2},M^{(1,\alpha)})=\delta_{1,\alpha}$ in categorical setting is equivalent to the assumption $U^{cc}=U$ in VOA coset construction setting. In all these counter examples to Kac-Wakimoto hypothesis $L_{\hat\p}(\dot k\lambda_0)^{cc}$ strictly larger than  $L_{\hat\p}(\dot k\lambda_0).$ This was also noticed in Section 3.1 of \cite{X3}.

To prove   Kac-Wakimoto Hypothesis we first establish a general result in category theory.
 First, we need some general results that will be used in the proof of the Kac-Wakimoto Hypothesis.  Given a modular tensor category ${\DD}$ and a regular commutative algebra $B\in \DD$ we denote the $B$-module category in $\DD$ by $\DD_B$.  Then $\DD_B$ is a fusion category. For any $M\in\DD_B$ we define a linear map $T_M: K(\DD_B)\to K(\DD_B)$ such that $T_M(N)=M\boxtimes_B N$ for any
 $N\in \Or(\DD_B)$ where $K(\FF)$ is the fusion algebra of $\FF$ for any fusion category $\FF.$
 Also set  $a_{\lambda}=B\boxtimes_{\DD}\lambda$ for  $\lambda\in \Or(\DD)$ and 
$\Or(\DD_B^{0})=\{\sigma_i|i\in\Delta\}.$  For short we denote $T_{\sigma_i}$ by $T_i$ and 
$T_{a_{\lambda}}$by $T_{\lambda}.$ Then $\{T_{i}, T_{\lambda} | i\in \Delta, \lambda\in \Or(\DD)\}$
commute with each other and can be diagonalized simultaneously  \cite{DMNO}.  It is well known that $\lambda\mapsto a_{\lambda}$ is an algebra homomorphism from $K(\DD)$ to $K(\DD_B).$ 
So $K(\DD_B)$ has a basis $v^{(i,\mu,m)}$ with $i\in \Delta$ and $\mu\in D$ such that
$$T_{j}v^{(i,\mu,m)}=\frac{s_{ji}}{s_{1i}}v^{(i,\mu,m)} $$
$$T_{\lambda}v^{(i,\mu,m)}= \frac{{s}^\DD_{\lambda\mu}}{{s}^\DD_{1\mu}} v^{(i,\mu,m)}$$
for all $j$ and $\lambda$ where $m$ is the index of basis vectors of the eigenspace with indicated eigenvalues. It is possible that for some $i,\mu,$ there are no eigenvectors $v^{(i,\mu,m)}.$ The key theorem is that there are  eigenvectors $v^{(i,\mu,m)}$ if and only if $\sigma^i$ is a direct summand of 
$a_{\mu}.$  This result plays an essential role in the proof of the Kac-Wakimoto Hypothesis and was obtained in the conformal net setting in Section 3.1 of \cite{X3}.

The paper is organized as follows. We give the basic materials on vertex operator algebras and their various modules in Section 2. We review the basics of the fusion category in Section 3. 
We give a categorical setting on coset construction in Section 4. We also define the Kac-Wakimoto set and discuss some properties of the Kac-Wakimoto set.  In Section 5 we prove that the space spanned by all $M^{(i,\alpha)}$ in $K(\CC_2)$ is invariant under the action of $S$-matrix.  We also give a dimension formula for $M^{(i,\alpha)}$ in terms of the Kac-Wakimoto set. In Section 6, we use the Kac-Wakimoto set to determine whether simple objects from $\CC_2$ appearing $M^i, M^j$ are the same, or there are no intersections.  Furthermore, we give several equivalent conditions on when $\KW$ forms a group. Section 7 is devoted to the proof of the Kac-Wakimoto hypothesis. We present more results on $M^{(i,\alpha)}$ in Section 8. We apply results on categorical cosets to VOA coset constructions in Section 9.

\section{Basics on vertex operator algebras}
\setcounter{equation}{0}
In this section, we review the basic concepts of vertex operator algebras including various notions of modules, rationality, and $C_2$-cofiniteness. We also discuss the commutant $U^c$ of a vertex operator subalgebra $U$  in $V$\cite{FZ}.

 Let $(V,Y,{\bf 1},\omega)$ be a vertex operator algebra (see \cite{B} and \cite{FLM}).
We first recall from \cite{DLM2} the definitions of weak module,
admissible module, ordinary module for a vertex operator algebra $V.$
 A {\em weak module}
 $M$ for $V$ is a vector space equipped with a linear map
$$\begin{array}{l}
V\to (\mbox{End}\,M)[[z^{-1},z]]\label{map}\\
v\mapsto\displaystyle{Y_M(v,z)=\sum_{n\in \Z}v_nz^{-n-1}\ \ (v_n\in\mbox{End}\,M)}
\mbox{ for }v\in V\label{1/2}
\end{array}$$
satisfying the following conditions for $u,v\in V$,
$w\in M$:
\begin{eqnarray*}\label{e2.1}
& &v_nw=0\ \ \  				
\mbox{for}\ \ \ n\in \Z \ \ \mbox{sufficiently\ large};\label{vlw0}\\
& &Y_M({\bf 1},z)=1;\label{vacuum}
\end{eqnarray*}
 \begin{equation*}\label{jacobi}
\begin{array}{c}
\displaystyle{z^{-1}_0\delta\left(\frac{z_1-z_2}{z_0}\right)
Y_M(u,z_1)Y_M(v,z_2)-z^{-1}_0\delta\left(\frac{z_2-z_1}{-z_0}\right)
Y_M(v,z_2)Y_M(u,z_1)}\\
\displaystyle{=z_2^{-1}\delta\left(\frac{z_1-z_0}{z_2}\right)
Y_M(Y(u,z_0)v,z_2)}.
\end{array}
\end{equation*}
This completes the definition. We denote this module by
$(M,Y_M)$ (or briefly by $M$).

An ({\em ordinary}) $V$-module is a  weak $V$-module which
carries a $\C$-grading
$$M=\bigoplus_{\lambda \in{\C}}M_{\lambda} $$
such that $\dim M_{\l}$ is finite and $M_{\l+n}=0$
for fixed $\l$ and $n\in {\Z}$ small enough. Moreover one requires that
$M_{\l}$ is the eigenspace for $L(0)$ with eigenvalue $\lambda:$
$$L(0)w=\l w=(\mbox{wt}\,w)w, \ \ \ w\in M_{\l}.$$

An {\em admissible} $V$-module is
a  weak $V$-module $M$ which carries a
${\Z}_{+}$-grading
$$M=\bigoplus_{n\in {\Z}_{+}}M(n)$$
($\Z_+$ is the set all nonnegative integers) such that if $r, m\in {\Z} ,n\in {\Z}_{+}$ and $a\in V_{r}$
then
$$a_{m}M(n)\subseteq M(r+n-m-1).$$
Note that any ordinary module is an admissible module.

A vertex operator algebra $V$ is called {\em rational} if any admissible
module is a direct sum of irreducible admissible modules. It was proved in
[DLM3] that if $V$ is rational then there are only
finitely many inequivalent irreducible admissible modules $V=M^0,..., M^p$
and each irreducible admissible module is an ordinary module. Each $M^i$ has weight space decomposition
$$M^i=\bigoplus_{n\geq 0}M^i_{\lambda_i+n}$$
where $\lambda_i\in\C$ is a complex number such that $M^i_{\lambda}\ne 0$ and $M^i_{\lambda_i+n}$ is the eigenspace of $L(0)$ with eigenvalue $\lambda_i+n.$ $\lambda_i$ is called the weight of $M^i.$

A vertex operator algebra $V$ is called $C_2$-cofinite if $\dim V/C_2(V)<\infty$ where
$C_2(V)$ is the subspace of $V$ spanned by $u_{-2}v$ for $u,v\in V$ \cite{Z}. If $V$ is both rational and
$C_2$-cofinite, then $\lambda_i$ and central charge $c$ are rational numbers \cite{DLM4}.

A vertex operator algebra $V$ is of {\em CFT type} if
$V$ is simple,  $V=\oplus_{n\geq 0}V_n$ and $V_0=\C {\bf 1}.$  $V$ is of strong CFT type if
$V$ further satisfies $L(1)V_1=0.$ It follows from \cite{L} that there is a unique nondegenerate symmetric
invariant bilinear from $\<,\>$ \cite{FHL} on $V$ such that $\<\1,\1\>=1.$ In particular,
the restriction of the form to $V_n$ is nondegenerate. 

\begin{lem} If $V$ is rational VOA of CFT type then $V$ is simple. 
\end{lem}
\begin{proof} Let $W$ be the maximal $V$-submodule of $V$ such that $W_0=W\cap V_0=0.$ Then $V=V^1\oplus W$ as $V$-module where $V^1$ is the submodule generated by
	${\bf 1}.$ It follows immediately that $V=V^1$ is irreducible $V$-module.
	\end{proof}

Let $V=(V,Y,{\bf 1},\omega)$ be a vertex operator algebra and $U=(U,Y, {\bf 1},\omega^1)$ is
a vertex operator subalgebra of $V.$ It is clear that $L^1(0)|_U=
L(0)|_U.$
The {\em commutant} $U^c$ of $U$ is defined to be
$$U^c=\{u\in V|v_nu=0, v\in U,n\geq 0\}$$
(cf. [FZ]).
Set $\omega^{2}=\omega-\omega^{1}$ and $Y(\omega^i,z)=\sum_{n\in \Z}L^i(n)z^{-n-2}.$ $U^c$ can be regarded
as the space of vacuum-like vectors  for $U$ [L1],
that is, $$U^c=\{u\in V|L^1(-1)u=0\}.$$

The following Lemma is well known (see \cite{FZ}, \cite{LL}).
\begin{lem}\label{imme} Let $V=(V,Y,{\bf 1},\omega)$ be a vertex operator algebra
and $U=(U,Y,{\bf 1},\omega^1)$ is a vertex operator subalgebra of $V$ such that $L(1)\omega^1=0.$

 (1) On any weak $V$-module, the actions of $U$ and $U^c$ are commutative.
That is,
$$Y(u,z_1)Y(v,z_2)=Y(v,z_2)Y(u,z_1)$$
 for $u\in U$ and $v\in U^c.$

(2) $U^c=(U^c,Y,{\bf 1},\omega^{2})$ is
also a vertex operator subalgebra of $V.$

(3) $U^{cc}=(U^c)^c\supset U$ and $U^{ccc}=U^c.$

(4) $U\otimes U^c$ is a subalgebra of $V.$
\end{lem}

Let $M=\oplus_{n\geq 0}M(n)$ be an irreducible admissible $V$-module which is completely reducible $U$-module.
Then
$$M=\bigoplus_{j\in J}W^j\otimes \Hom_{U}(W^j,M)$$
where $W^j$ are the irreducible $U$-modules occurring in $M$ and $\Hom_{U}(W^j,M)$ is the space of $U$-module homomorphism from $W^j$ to $M.$

 \begin{lem}\label{l6.1} Let $M$ and $W^j$ be as above. Let $0\ne x\in W^j$ be a homogeneous vector. Then $f\mapsto f(x)$ is an injective linear map from $\Hom_{U}(W^j,M)$ to $M.$ That is, we can identify $\Hom_{U}(W^j,M)$ with its image $M_j$ in $M.$
\end{lem}
\begin{proof} Since $W^j$ is an irreducible $U$-module, the result is immediate.
\end{proof}

\begin{lem}\label{l6.2} Let $M$ and $W^j$ be as above.  Then $\Hom_{U}(W^j,M)$ is a $U^c$-module such that for any $v\in U^c,$ $n\in\Z$  and $f\in \Hom_{U}(W^j,M),$ $(v_nf)(w)=v_nf(w)$ for all $w\in W^j.$ Moreover, $M_j$ is a $U^c$-submodule of $M$ and the identification in Lemma \ref{l6.1} is a $U^c$-module isomorphism.
\end{lem}

\begin{proof} We first prove that $v_nf\in \Hom_{U}(W^j,M)$ for $v\in U^c,$ $m\in\Z$  and $f\in \Hom_{U}(W^j,M).$ Let $u\in U$ and $n\in \Z.$ Then $(v_mf)(u_nx)=v_mu_nf(x)=u_nv_mf(x)=u_n(v_mf)(x)$
where we have use the fact that $u_n,v_m$ commute on $M.$ The other module axioms is easy to check as
$v_n,$ in fact, acts on $M.$

To show $M_j$ is a $U^c$-submodule of $M$ we need to verify that $v_mM_j\subset M_j$ for
$v\in U^c$ and $m\in\Z.$ Again, let $f\in \Hom_{U}(W^j,M).$ Then $v_mf(x)=(v_mf)(x)$ is the image
of the homomorphism $v_mf.$ The isomorphism from $\Hom_{U}(W^j,M)$ to $M_j$ is clear.
In particular,  we have the following decomposition
$M=\oplus_{j\in J}W^j\otimes M^j$  as $U\otimes U^c$-module.
\end{proof}

The following result does not require that $U$ or $V$ is rational.  
\begin{prop}\label{np6.3} Assume that $V$ is a simple vertex operator algebra, $U=(U,Y,\1,\omega^1)$ is a simple vertex operator subalgebra of $V$ such that $L(1)\omega^1=0$ and $V$ is a completely reducible $U$-module.
Then $U^c$ is a simple vertex operator algebra.
\end{prop}
\begin{proof}
 From the discussion above  we have the decomposition
$$V=\bigoplus_{i\in I}U^i\otimes (U^c)^i$$
as $U\otimes U^c$-modules where $U^i$ are inequivalent $U$-modules occurring in $V$ and $(U^c)^i$ is the multiplicity space of $U^i$ in $V.$ We also assume $0\in I.$

Since $V$ is simple,   we know that for any nonzero $v\in U\otimes U^c,$
$$V=\<u_nv|u\in V,n\in\Z\>=\sum_{i\in I}\<u_nv|u\in U^i\otimes (U^c)^i,n\in\Z\>.$$
Clearly, $u_nv\in U^i\otimes (U^c)^i$ if $u\in U^i\otimes (U^c)^i.$ Thus
$U\otimes U^c$ is spanned by $u_nv$ for $u\in U\otimes U^c$ and $n\in\Z.$ Since $v$ is arbitrary
we see that $U\otimes U^c$ is a simple vertex operator algebra. Consequently, $U^c$ is a simple vertex operator algebra, as expected.
\end{proof}

\section{Basics on fusion categories }

We investigate categorical coset theory in the next few sections.  These results will be used later to study the rational coset theory for vertex operator algebra.

We first recall some basics of category theory from \cite{KO}, \cite{EGNO}.  An object $A$ in a braided fusion category $\CC$ is called regular commutative algebra  if there are morphisms
$\mu: A\boxtimes A\to A$ and $\eta: {\bf 1_\CC}\to A$ such that 
$\mu\circ(\mu\boxtimes {\rm id}_A)\circ \alpha_{A,A,A} = \mu\circ({\rm id}_A\boxtimes\mu)$,  $\mu\circ(\eta\boxtimes {\rm id}_A)\circ l_A^{-1}={\rm id}_A =\mu\circ({\rm id}_A\boxtimes\eta)\circ r_A^{-1},$
$\mu = \mu\circ c_{A,A}$
and
$\dim\Hom(\1_\CC,A)=1$ where $\alpha_{A,A,A}: A\boxtimes (A\boxtimes A)\to (A\boxtimes A)\boxtimes A$ is the associative isomorphism,  $l_A: \1_\CC\boxtimes A\to A$ is the left unit isomorphism, $r_A: A\boxtimes \1_\CC\to A$ is the right unit isomorphism and $c_{A,A}: A\boxtimes A\to A\boxtimes A$ is the braiding isomorphism. A left $A$-module $N$ is an object in $\CC$ with a morphism $\mu_N: A\boxtimes N\to N$ such that $\mu_N\circ (\mu\boxtimes \id_N)\circ \alpha_{A,A,N}=\mu_N\circ (\id_A\boxtimes \mu_N).$ We denote the left $A$-module category by $\CC_A.$ Let $N_{1}, N_{2} \in \CC_{A}$. Define $N_1\boxtimes_A N_2=N_1\boxtimes N_2/\mathrm{Im}(\mu_1-\mu_2)$ where $\mu_1, \mu_2: A\boxtimes N_1 \boxtimes N_2 \rightarrow N_1 \boxtimes N_2$ are defined by $\mu_1=\mu_{N_1}\boxtimes \mathrm{id}_{N_2}$, $\mu_2=(\mathrm{id}_{N_1}\boxtimes \mu_{N_2})\circ c_{A, N_1}\boxtimes \id_{N_2}$. 
Then  $\CC_A$ is a fusion category with tensor product $N_1\boxtimes_A N_2$. An $A$-module $N$ is called local if
$\mu_N\circ c_{N,A}\circ c_{A,N}=\mu_N$.  We denote the local $A$-module category by $\CC_A^0.$ Then  $\CC_A^0$ is a braided fusion category. Moreover, if $\CC$ is modular tensor category, so is $\CC_A^0$ \cite{KO}. For 
any $N\in \CC_A$ and $X\in\CC,$ $N\boxtimes X\in \CC_A.$  

Let $\CC$ be a fusion category whose Grothendieck ring is denoted by $K_0(\CC).$ Then there is a unique ring homomorphism $\FPdim: K_0(\CC)\to \BR$ satisfying
$\FPdim(M)\geq 1$ for any nonzero object $M.$ The Frobenius-Perron dimension of $\CC$ is defined to be $\FPdim(\CC)=\sum_{M\in \Or(\CC)}\FPdim (M)^2.$ In the case   $\CC$ is a fusion subcategory of the module category for
a vertex operator algebra $V$, the Frobenius-Perron dimension $\FPdim (M)$ is exactly the quantum dimension $\qdim_V(M)$ studied in \cite{DJX} and \cite{DRX}. For short we set
$\<X,Y\>=\dim\Hom(X,Y)$ for $X,Y\in\CC.$

Now assume that $\CC$ is a braided fusion category and $A$  is a regular commutative algebra in $\CC.$ The following identities give relations among dimensions of relevant categories and algebra:
$$\FPdim(\CC_A) =\frac{\FPdim(\CC)}{\FPdim(A)},\  \FPdim(\CC_A^0) =\frac{\FPdim(\CC)}{\FPdim(A)^2}$$
where the first identity holds for any braided fusion category $\CC$ and the second identity requires that $\CC$ is modular \cite{DMNO}.

From \cite{H}  we know that the module category ${\cal C}_{V}$ of any rational, $C_2$-cofinite, selfdual  vertex operator algebra $V$ is a modular tensor category with tensor product $\boxtimes.$ If $A$ is an extension of $V$ then $A$ is also rational, $C_2$-cofinite
\cite{ABD, HKL}. Moreover $A$ is  a regular commutative algebra in $\CC_{V},$   $(\CC_{V})_A$  is a fusion category and $(\CC_{V})_A^0$ is exactly $\CC_A$ \cite{HKL}.

\section{Categorical coset theory}

Let $\CC_1, \CC_2$ be modular tensor categories. Let $ \Or(\CC_1)=\{W^\alpha|\alpha\in J\}$ and $\Or(\CC_2)=\{N^{\phi}|\phi\in K\}$ with  $W^1=1_{\CC_1}$ and $N^1=1_{\CC_2}$ 
where we assume that $1\in J, 1\in K.$ Let $A$ be a regular commutative algebra in $\CC_1\otimes \CC_2$. Then $\CC=(\CC_1\otimes \CC_2)_A^0$ is also a modular tensor category. Let  $\Or(\CC)=\{M^i|i\in I\}$ with $1\in I$ and $M^1=A.$ Then $M^i\cong\oplus_{\alpha\in J}W^\alpha\otimes M^{(i,\alpha)}$ 
as objects in $\CC_1\otimes \CC_2$ for $i\in I.$  For short we  identify $\alpha, \phi ,  (i,\alpha) $ with $W^{\alpha}, N^{\phi}, M^{(i,\alpha)},$ respectively. Also, set 
$J_i=\{\alpha \mid (i,\alpha)\neq 0 \}.$

We assume the following in this paper:

(1) $\CC_1$ and $\CC_2$ are pseudo unitary. That is, for any $X$ in $\CC_1,\CC_2,$ $\FPdim(X)=\dim (X)$ where $\dim(X)$ is the categorical dimension of $X.$ 

(2) $M^{(1,1)}=1_{\CC_2}$ and $\Hom_{\CC_2}(1_{\CC_2},M^{(1,\alpha)})=\delta_{1,\alpha}.$ 

For $\alpha\in J,$  $\phi\in K$  let  
$$a_{\alpha\otimes \phi}=A\boxtimes_{\CC_1\otimes \CC_2} (\alpha\otimes \phi), $$
$$\ a_{\alpha\otimes 1}=a_{\alpha\otimes 1_{\CC_2}}$$
and 
$$a_{1\otimes (i,\alpha)}=A\boxtimes_{\CC_1\otimes \CC_2}  (1_{\CC_1}\otimes (i,\alpha )).$$
Since $A$ is a regular commutative algebra in $\CC_1\otimes\CC_2,$ we know that  $a_{\alpha\otimes \phi}, a_{1\otimes (i,\alpha)}\in (\CC_1\otimes\CC_2)_A.$

\begin{lem} \label{a} 
	Both $a_{\alpha\otimes 1}$ and $a_{1\otimes \phi}$ are  simple objects in $(\CC_1\otimes \CC_2)_A.$ That is, $\<a_{\alpha\otimes 1},a_{\alpha\otimes 1}\>=\<a_{1\otimes \phi},a_{1\otimes \phi}\>=1$.
\end{lem}
\begin{proof} Since $(\CC_1\otimes \CC_2)_A$ is a fusion category we have
	$$\Hom_A(a_{\alpha\otimes 1},a_{\alpha\otimes 1})=\Hom_{\CC_1\otimes \CC_2}(\alpha\otimes 1,a_{\alpha\otimes 1}).$$
	Note that 
	\begin{eqnarray*}
		a_{\alpha\otimes 1} &=&A\boxtimes_{\CC_1\otimes \CC_2}  (\alpha\otimes 1)\\
		&=&(\oplus_{\beta\in J_1}(\beta \otimes {(1,\beta)})\boxtimes_{\CC_1\otimes
			\CC_2}({\alpha}\otimes 1)\\
		& =&\oplus_{\beta\in J_1} ({\beta}\boxtimes {\alpha})\otimes {(1,\beta)}\\
		& =&\alpha\otimes 1 \oplus (\oplus_{\beta \neq 1}({\beta}\boxtimes {\alpha})\otimes (1,\beta)).
	\end{eqnarray*}
	Using the assumption  shows that  
	$$1=\dim \Hom_{\CC_1\otimes \CC_2}({\alpha}\otimes 1,a_{\alpha\otimes 1})=\dim \Hom_A(a_{\alpha\otimes 1},a_{\alpha\otimes 1}),$$
	as expected. Similarly for $a_{1\otimes \phi}.$
\end{proof}	

Let $\FF^1$ be the fusion subcategory of $\CC_1$ generated by $\alpha\in J_1$ and $\FF^2$ be the fusion subcategory of $\CC_2$ generated by the simple objects appearing in $(1,\alpha).$ 
The following Theorem was essentially obtained in \cite{Lin} with a similar proof.
\begin{thm}\label{lin} Let $\alpha, \beta\in J_1.$ Then
	
	(1)  $(1,\alpha)$ is simple,
	
	(2)  $(1,\alpha)$ and $(1.\beta)$ are isomorphic  iff $\alpha=\beta,$

 (3) $\dim \alpha=\dim (1,\alpha),$
	
(4) If $\alpha\boxtimes \beta =\sum_{\gamma\in J}N_{\alpha \beta}^{\gamma} {\gamma}$
	and  $N_{\alpha \beta}^{\gamma}\neq 0$, then $\gamma \in J_1$.
	That is, $\Or(\FF^1)=\{W^{\alpha}\mid \alpha \in  J_1\},$ 
	
	(5) There is a braided equivalent functor $F:\FF^1\to \overline{\FF^2}$ such that
	$F(W^{\alpha})$ is isomorphic to $(M^{(1,\alpha)})'$ where 
 $\overline{\FF^c}$ is the reverse category of $\FF^c$ and $(M^{(1,\alpha)})'$ is the dual of $M^{(1,\alpha)}.$ In particular, $(1,\alpha)'\boxtimes (1,\beta)'=\sum_{\gamma\in J_1}N_{\alpha \beta}^{\gamma}(1,\gamma)'.$
\end{thm}

\begin{lem}\label{fusion}
$\FF_i$ is an indecomposable left module category over $\FF^1$-module. In particular, $W^\beta\boxtimes W^{\alpha}=\sum_{\gamma\in J_i}N_{\beta\alpha}^{\gamma}W^{\gamma}$ for all $\beta\in J_1$ and $\alpha\in J_i.$
\end{lem} 
We now define the Kac-Wakimoto set
$$\KW=\{\alpha\in J|a_{\alpha\otimes 1}\in \CC\}.$$
The importance of $\KW$-set was first noticed in \cite{KW} when they studied the coset vertex operator superalgebras associated with affine vertex operator superalgebras. It turns out that the KW-set plays an essential role in understanding how to decompose  $M^{(j,\beta)}$ into 
simple objects in $\CC_2.$

\begin{lem}\label{kwl1} Assume that $\alpha\in \KW.$ Then 

	(1) $\alpha'\in \KW,$
	
	(2)  $\<1_{\CC_2},(i,\alpha)\>=1$
where $M^i=a_{\alpha\otimes 1},$

(3) if $\beta\in \KW$ such that $a_{\alpha\otimes 1}=a_{\beta\otimes 1} ,$  then $\alpha=\beta.$ 
\end{lem}
\begin{proof}
	
	(1) Clearly,  $a_{\alpha'\otimes 1}=(a_{\alpha\otimes 1})'\in \CC$   by noting that $(1_{\CC_2})'=1_{\CC_2},$ 	the result follows.
	
	(2) 
Since  $a_{\alpha \otimes 1}$ is a  simple $A$-modules in $(\CC_1\otimes\CC_2)_A$   we have 
	$$1=\<a_{\alpha \otimes 1},a_{\alpha \otimes 1}\>_A=\dim \Hom_{\CC_1\otimes \CC_2}(\alpha \otimes 1,a_{\alpha \otimes 1})
	=\dim\Hom_{\CC_1\otimes \CC_2}(\alpha \otimes 1,\alpha\otimes(i,\alpha)).$$  
	This implies that $\<1_{\CC_2},(i,\alpha)\>=1.$

(3)  As $\Hom_{A}(a_{\alpha \otimes 1},a_{\beta\otimes 1})
		\cong\Hom_{\CC_1\otimes \CC_2}({\alpha}\otimes 1,a_{\beta\otimes 1})
		\cong\Hom_{\CC_1\otimes \CC_2}({\alpha}\otimes 1, \beta\otimes 1)$
	is one-dimensional,  we conclude that $\alpha=\beta$.
\end{proof}	


\begin{lem}\label{sigma}  Let $\alpha,\beta\in \KW.$
	
	(1) $\alpha\boxtimes_{\CC_1} \beta=\sum_{\gamma\in \KW} N_{\alpha \beta}^{\gamma} \gamma.$ 
	
	(2) $a_{\alpha\otimes 1}\boxtimes_A a_{\beta\otimes 1} =\sum_{\gamma\in\KW} N_{\alpha \beta}^{\gamma}a_{\gamma\otimes 1} $.
	
	(3) Let $\CC_1^{kw}$ be the braided fusion subcategory of $\CC_1$ generated by $\alpha\in \KW$ and  $\CC^{kw}$ be the braided fusion subcategory of $\CC$ generated by $a_{\alpha\otimes 1}$ for $\alpha\in \KW.$ Then $\Or(\CC_1^{kw})=\KW,$ 
	$\Or(\CC^{kw})=\{a_{\alpha\otimes 1}| \alpha\in \KW\},$ and $\CC_1^{kw}$ and $\CC^{kw}$ are braided equivalent.
	\end{lem}

\begin{proof}
	Note that 
$$a_{\alpha\otimes 1} \boxtimes_A a_{\beta\otimes 1}
\cong a_{(\alpha\boxtimes \beta)\otimes 1}
=\sum_{\gamma\in J } N_{\alpha, \beta}^{\gamma} a_{\gamma \otimes 1}.$$
Since $a_{\alpha\otimes 1},  a_{\beta\otimes 1}\in \CC$ which is a modular tensor category, we see immediately that $a_{\gamma \otimes 1}\in \CC$ if  $N_{\alpha \beta}^\gamma\ne 0.$ This proves both (1) and (2).

For (3),  we set  ${\cal F}(X)=A\boxtimes (X\otimes 1_{\CC_2})$ for $X\in\CC_1.$ From \cite[Theorem 2.67]{CKM},  ${\cal F}: \CC^{kw}_1\to \CC^{kw}$
is a braided tensor  functor by noting  that $X\to X\otimes 1_{\CC_2}$ 
is a braided tensor functor.  The braided equivalence follows from (1) and (2)  immediately.
\end{proof}

 We now finger out which  $\alpha\in J$ lies in $\KW.$   Let ${\DD}=\{W^\beta|\beta\in J_1\}.$  The \emph{M\"uger centralizer} $C_{\CC_1}(\DD)$ is the  subcategory of $\CC_1$ consisting of the objects $Y$ in $\CC_1$ such that $c^1_{Y,X}\circ  c^1_{X,Y} = \id_{X\boxtimes Y}$  for all $X$ in $\DD$ where $c^1_{Y,X}:Y\boxtimes X\to X\boxtimes Y$ is the braiding isomorphism. Let $\theta^1, \theta^2,\theta$ be the ribbon structures on categories
$\CC_1,\CC_2,\CC.$ Then $\theta^1\otimes \theta^2$ is a ribbon structure on $\CC_1\otimes\CC_2.$

\begin{prop}\label{characwk} Assume that $(\theta^1\otimes\theta^2 )_A=\id.$  Let $\alpha\in J.$ Then $\alpha\in \KW$ if and only if $\alpha\in C_{\CC_1}(\DD).$ Equivalently, $a_{\alpha\otimes 1}\in \CC$  if and only if $W^\alpha\in C_{\CC_1}(\DD).$
\end{prop}
\begin{proof}  Clearly, $a_{\alpha\otimes 1}$ is a simple object in $(\CC_1\otimes \CC_2)_A.$ By \cite[Theorem 3.3]{KO}, $a_{\alpha\otimes 1}\in \CC$ if and only if $(\theta^1
\otimes \theta^2)_{a_{\alpha\otimes 1}}$ is a constant. From the definition, 
$$a_{\alpha\otimes 1}=\oplus_{\beta\in J_1}(\beta\boxtimes \alpha)\otimes (1,\beta).$$
Since $(\theta^1\otimes \theta^2)_{{\alpha}\otimes 1}=\theta^1_{\alpha}$ where  $\theta^1_{\alpha}=\theta^1_{W^{\alpha}},$  $a_{\alpha\otimes 1}\in \CC$ if and only if
$$(\theta^1
\otimes \theta^2)_{({\beta}\boxtimes {\alpha})\otimes (1,\beta)}=
\theta^1_{\beta\boxtimes \alpha}\otimes \theta^2_{(1,\beta)}=\theta^1_{\alpha}\id_{({\beta}\boxtimes {\alpha})\otimes (1,\beta)}$$
 where we have identified $\theta^1_{\alpha}$ with a complex number as $\theta^1_{\alpha}$ acts on $W^{\alpha}$ is a constant. 
On the other hand, 
$$\theta^1_{{\beta}\boxtimes {\alpha}}=c_{W^{\alpha},W^{\beta}}^1\circ c_{W^{\beta},W^{\alpha}}^1\circ \theta^1_{\beta}\boxtimes \theta^1_{\alpha}= \theta^1_{\beta}\theta^1_{\alpha}c_{W^{\alpha},W^{\beta}}^1\circ c_{W^{\beta},W^{\alpha}}^1$$
where we again regard  $\theta^1_{\beta},\theta^1_{\alpha}$ as constant. From the assumption that $(\theta^1\otimes\theta^2 )_A=\id$ we know that $(\theta^1\otimes \theta^2)_{{\beta}\otimes (1,\beta)}=\theta^1_{\beta}\theta^2_{(1,\beta)}=1.$
Thus 
$$(\theta^1
\otimes \theta^2)_{({\beta}\boxtimes {\alpha})\otimes (1,\beta)}=\theta^1_{\alpha}
(c_{W^{\alpha},W^{\beta}}^1\circ c_{W^{\beta},W^{\alpha}}^1)\otimes \id_{(1,\beta)}=\theta^1_{\alpha}\id_{({\beta}\boxtimes {\alpha})\otimes (1,\beta)}$$
if and only if $c_{W^{\alpha},W^{\beta}}^1\circ c_{W^{\beta},W^{\alpha}}^1=\id_{\beta\boxtimes \alpha}.$  The proof is complete.
\end{proof}

We remark that in the setting of vertex operator algebra, the commutative algebra $A$ in a module category of a regular vertex operator algebra $V$ is also a vertex operator algebra 
\cite{HKL}. The assumption $\theta_A=1$ in Proposition \ref{characwk}   is always true.

We next deal with $(i,\alpha)$ for $\alpha\in J_i.$ For short, we set $d_X=\dim X$ for $X\in {\DD}$ where $\DD$ is any fusion category. We  also set $d_i=d_{M^i}$ and $d_{\alpha}=d_{W^\alpha}$ for $i\in I, \alpha\in J.$
\begin{lem}\label{kwl2} 
The $a_{1\otimes (i,\alpha)}$ is a summand of $M^i\boxtimes_A a_{\alpha'\otimes 1}$ in
	$(\CC_1\otimes \CC_2)_A.$  Moreover, $a_{1\otimes (i,\alpha)}=M^i\boxtimes_A a_{\alpha'\otimes 1}$ if $d_{(i,\alpha)}=d_id_\alpha.$ 
	
	\end{lem}
\begin{proof}
	
	One can verify that 
	\begin{eqnarray*}
		\Hom_{A}(a_{1\otimes (i,\alpha)},a_{1\otimes (i,\alpha)})
		&=&\Hom_{A}(A\boxtimes_{\CC_1\otimes \CC_2}1_{\CC_1}\otimes (i,\alpha ),a_{1\otimes (i,\alpha)})\\
		&\cong&\Hom_{\CC_1\otimes \CC_2}(1_{\CC_1}\otimes (i,\alpha ),a_{1\otimes (i,\alpha)})\\
		&=&\Hom_{\CC_1\otimes \CC_2}(1_{\CC_1}\otimes (i,\alpha ), A\boxtimes_{\CC_1\otimes \CC_2}(1_{\CC_1}\otimes (i,\alpha )))\\
		&=&\Hom_{\CC_1\otimes \CC_2}(1_{\CC_1}\otimes (i,\alpha ),(\sum_{\beta\in  J_1}{\beta}\otimes {(1,\beta)})\boxtimes_{\CC_1\otimes \CC_2}(1_{\CC_1}\otimes (i,\alpha )))\\
		&=&\Hom_{\CC_1\otimes \CC_2}(1_{\CC_1}\otimes (i,\alpha ),1_{\CC_1}\otimes (i,\alpha ))\\
		&=&\Hom_{\CC_2}({(i,\alpha )},{(i,\alpha )}).
	\end{eqnarray*}
	Similarly, 
	\begin{eqnarray*}
		\Hom_{A}(a_{1\otimes (i,\alpha)},M^i\boxtimes_A a_{\alpha'\otimes 1})
		&=&\Hom_{A}(A\boxtimes_{\CC_1\otimes \CC_2}(1_{\CC_1}\otimes(i,\alpha )),M^i\boxtimes_A (A\boxtimes_{\CC_1\otimes \CC_2}{\alpha'}\otimes 1_{\CC_2}))\\
		&\cong&\Hom_{\CC_1\otimes \CC_2}(1_{\CC_1}\otimes (i,\alpha ),M^i\boxtimes_{\CC_1\otimes \CC_2} ({\alpha'}\otimes 1_{\CC_2}))\\
		&\cong&\Hom_{\CC_1\otimes \CC_2}((1_{\CC_1}\otimes (i,\alpha ))\boxtimes ({\alpha}\otimes 1_{\CC_2}),M^i)\\
		&=&\Hom_{\CC_1\otimes \CC_2}({\alpha}\otimes (i,\alpha ),M^i)\\
		&=&\Hom_{\CC_1\otimes \CC_2}({\alpha}\otimes (i,\alpha ),{\alpha}\otimes (i,\alpha))\\
		&=&\Hom_{\CC_2}({(i,\alpha )},{(i,\alpha)}).
	\end{eqnarray*}
	
	Let $(i,\alpha)=\sum _j m_j X_j$, where $X_j$ are inequivalent  simple objects in $\CC_2 ,$ and  $m_j$ are the multiplicities. Then $a_{1\otimes (i,\alpha)}=\sum_j m_j a_{1\otimes X_j }$.
	It is good enough to prove $m_j a_{1\otimes X_j }$ is contained in $M^i\otimes a_{\alpha' \otimes 1}$ or equivalently $\dim \Hom ( a_{1\otimes X_j },M^i\otimes a_{\alpha' \otimes 1})=m_j.$ The calculation  
	\begin{eqnarray*}
	\<a_{1\otimes X_j}, M^i\boxtimes a_{\alpha' \otimes 1}\>
		&=&\<1_{\CC_1}\otimes X_j, M^i\boxtimes a_{\alpha' \otimes 1}\>\\
		&=&\<1_{\CC_1}\otimes X_j, M^i \boxtimes {\alpha'}\otimes 1_{\CC_2}\>\\
		&=&\<{\alpha}\otimes X_j, M^i\>\\
		&=&\<X_j,(i,\alpha)\>\\
		&=&m_j,
	\end{eqnarray*}
	gives the desired result.
	
	Note that $d_{(i,\alpha)}=d_{1\otimes (i,\alpha)}.$ So if  $d_{(i,\alpha)}=d_id_\alpha$, then  $d_{1\otimes (i,\alpha)}=d_id_\alpha=d_{M^i\boxtimes_A a_{\alpha'\otimes 1}}$ and   $a_{1\otimes (i,\alpha)}=M^i\boxtimes_A a_{\alpha'\otimes 1}.$ The proof is complete.
	\end{proof}

\begin{rem} \label{r4.8}It is worthy to mention that Lemma \ref{kwl2} holds for any $M\in \Or((\CC_1\otimes \CC_2)_A).$
	\end{rem}

The following result tells us why the $\KW$ set is important.
\begin{prop}\label{p3.6}
	Let $i,j\in I$ and $\alpha,\beta\in J.$ Then 
	$$	\<M^i\boxtimes_A a_{\alpha'\otimes 1},M^j\boxtimes_A a_{\beta'\otimes 1}\>=\sum_{\epsilon\in \KW}N_{\alpha', \beta}^{\epsilon} N_{i',j}^k$$ 
	where $M^k=a_{\epsilon\otimes 1}.$ 
	In particular, if $\KW=\{1_{\CC_1}\}$ then $(i,\alpha)$ is simple for all $i\in I$ and $\alpha\in J_i.$
\end{prop}
\begin{proof}
	 We have 
	\begin{eqnarray*}
		\<M^i\boxtimes_A a_{\alpha'\otimes 1},M^j\boxtimes_A a_{\beta'\otimes 1}\>
		&=&\< a_{\alpha'\otimes 1}\boxtimes_A a_{\beta\otimes 1},M^{i'}\boxtimes_A M^j\>\\
		&=&\< a_{(\alpha\boxtimes_{\CC_1} \beta')\otimes 1},M^{i'}\boxtimes_A M^j\>\\
		&=&\sum_{\gamma}\sum_{k}N_{\alpha', \beta}^{\gamma} N_{i',j}^k\< a_{\gamma\otimes 1},M^k\>\\
		&=&\sum_{\gamma}\sum_{k}N_{\alpha, \beta'}^{\gamma} N_{i',j}^k\< \gamma\otimes 1_{\CC_2},M^k\>\\
		&=&\sum_{\gamma}\sum_{k}N_{\alpha', \beta}^{\gamma} N_{i',j}^k\< \gamma\otimes 1_{\CC_2},\oplus_{\epsilon}\epsilon\otimes (k,\epsilon)\>\\
		&=&\sum_{\gamma}\sum_{\epsilon\in \KW}N_{\alpha', \beta}^{\gamma} N_{i',j}^k
		\<\gamma\otimes 1_{\CC_2}, \epsilon\otimes (k,\epsilon)\>\\
		&=&\sum_{\epsilon\in \KW}N_{\alpha', \beta}^{\epsilon} N_{i',j}^k
		\<\epsilon \otimes 1_{\CC_2}, \epsilon\otimes (k,\epsilon)\>\\
		&=&\sum_{\epsilon\in \KW}N_{\alpha, \beta'}^{\epsilon} N_{i',j}^k
		\end{eqnarray*}	
	where we have used $M^k$ for $a_{\epsilon\otimes 1}$ in the last three equations involving KW set.
	
If $\KW=\{1_{\CC_1}\}$ then $\<M^i\boxtimes_A a_{\alpha'\otimes 1},M^j\boxtimes_A a_{\beta'\otimes 1}\>=N_{\alpha', \beta}^{1} N_{i',j}^1=\delta_{\alpha,\beta}\delta_{i,j}.$
This implies that $M^i\boxtimes_A a_{\alpha'\otimes 1}$  is simple for $\alpha\in J_i,$ and $M^i\boxtimes_A a_{\alpha'\otimes 1},$ $M^j\boxtimes_A a_{\beta'\otimes 1}$ are inequivalent for $\beta\in J_j$ and $(i,\alpha)\ne (j,\beta).$
By Lemma \ref{kwl2},
$a_{1\otimes (i,\alpha) }=M^i\boxtimes_A a_{\alpha'\otimes 1}$ is simple. Thus 
$\{(i,\alpha)|i\in I,\alpha\in J_i\}$ are all the inequivalent simple objects in $\CC_2.$
\end{proof}
\begin{rem}\label{r4.10a} Using Lemma \ref{kwl2} and Proposition \ref{p3.6} we see that 
if $N_{\alpha,\alpha'}^{\epsilon}=\delta_{\epsilon, 1}$ for all $\epsilon\in\KW$ then
$(i,\alpha)$ is a simple in $\CC_2.$
\end{rem}
\begin{rem}\label{r4.10} If we replace $M^i,M^j$ by arbitrary simple $A$-modules in $(\CC_1\otimes \CC_2)_A$ in Proposition \ref{p3.6}, the sum on the right hand side will be over $\epsilon \in J.$
	\end{rem}

Our goal is to decompose $(i,\alpha)$ into a direct sum of simple objects in $\CC_2.$ Proposition \ref{p3.6} gives an upper bound of $\dim \Hom((i,\alpha),(i,\alpha))$
for $\alpha\in J_i.$ If $d_{(i,\alpha)}=d_id_\alpha,$ then we know $\dim \Hom((i,\alpha),(i,\alpha))$ precisely. Even in this case, there is still a distance to determine the exact decomposition of $(i,\alpha)$ into a direct sum of simple objects. 
	
\section{$S$-matrices and $\KW$-set}

In this section,  we determine how the normalized $s$-matrix $\ddot{s}$ of category $\CC_2$ acts  acts on $(i,\alpha)$ for $i\in I$ and $\alpha\in J_i$ and show that the subspace of $K_2$ spanned by $(i,\alpha)$ is invariant under the action of $\ddot{s}.$    We also give a formula for $\dim (i,\alpha)$ using the $\KW$ set.

We first discuss the relations among  $S$-matrices in categories $\CC_1,\CC_2$ and $\CC=(\CC_1\otimes\CC_2)_A^0.$ 
Let $c^1_{\alpha,\beta}: W^{\alpha}\boxtimes W^{\beta} \to W^{\beta}\boxtimes W^{\alpha}$
be the braiding in $\CC_1,$ $c^2_{\phi, \xi}: N^{\phi}\boxtimes N^{\xi}\to N^{\xi}\boxtimes N^{\phi}$ be the braiding in $\CC_2$ and $c_{i,j}:M^i\boxtimes M^j\to M^j\boxtimes M^i$ be the braiding in $\CC.$ Let $\theta^1, \theta^2,\theta$ be the ribbon structures on categories
$\CC_1,\CC_2,\CC.$ Then $\theta^i_{X\boxtimes Y}=\theta_X^i\boxtimes \theta_Y^i\circ c^i_{Y,X}\circ c^i_{X,Y}$ for $X,Y\in \CC_i.$ Similar relation holds for $\theta.$ 
Set $\tilde{\dot{s}}_{\alpha\beta}=\tr (c^1_{W^\beta,W^\alpha}\circ c^1_{W^\alpha,W^\beta}).$ Let  $\tilde{\dot{s}}=(\tilde{\dot{s}}_{\alpha\beta}), \tilde{\ddot{s}}=(\tilde{\ddot{s}}_{\phi\xi}), 
\tilde{s}=(\tilde{s}_{ij})$ be the $S$-matrices associated to the modular tensor categories $\CC_1, \CC_2, \CC.$   We also set $\dot{s}=\frac{\tilde{\dot{s}}}{D_1}, \ddot{s}=\frac{\tilde{\ddot{s}}}{D_2},
s=\frac{\tilde{s}}{D}$ where $D_1=\sqrt{\dim \CC_1}, D_2=\sqrt{\dim \CC_2}, D=\sqrt{\dim \CC}.$ Also let $K_i=K(\CC_i)$ be the fusion algebra of $\CC_i$ for $i=1,2.$ Then the fusion algebra of $\CC_1\otimes\CC_2$ is $K_1\otimes K_2.$ Also let $K(\CC)$ be the fusion algebra of $\CC.$ Then $\tilde{\dot{s}},$   $\tilde{\ddot{s}},$ $\tilde{s}$ act on $K_1, K_2, K(\CC)$ respectively.  Let $G: \CC\to \CC_1\otimes \CC_2$ be the functor of restriction. Then $G$ induces a linear map from $K(\CC)$ to $K_1\otimes K_2$ such that 
$$G(dim A)\tilde{s}=(\tilde{\dot{s}}\otimes \tilde{\ddot{s}})G$$
by  \cite[Theorem 4.1]{KO}. Note that $\dim (\CC_1\otimes \CC_2)=D_1^2D_2^2$ and $D=\frac{D_1D_2}{\dim A}.$ The following result is immediate.
\begin{lem}\label{s-matrix} We have 
	$$Gs=(\dot{s}\otimes \ddot{s})G.$$
\end{lem}

We now determine the action of $\ddot{s}$ on each $(i,\alpha)\in K_2$ for $i=1,...,p.$  Note that $(i,\alpha)=0$
if $\alpha\not\in J_i.$ Let $\dot{S}, \ddot{S}, S$ be the corresponding linear transformations on
$K_1,K_2,K(\CC).$ Define $q\times p$ matrix $Z=((i,\alpha))_{\alpha=1,...,q, i=1,...,p}\in M_{p\times q}(K(\CC_2))$ with 
entries in $K_2$ 
and row vectors
$$\overrightarrow{M}=(M^i)_{i=1,...,p}\in K(\CC)^p,\ \ \ \ \ \overrightarrow{W}=(W^\alpha)_{\alpha=1,...,q}\in K_1^q.$$
It is easy to see that 
$G\overrightarrow{M}=\overrightarrow{W}\otimes Z.$ Note that
$$S\overrightarrow{M}=(SM^1,...,SM^p)=\overrightarrow{M}s,$$
$$\dot{S}\overrightarrow{W}=(\dot{S}W^1,...,\dot{S}W^q)=\overrightarrow{W}\dot{s},$$
$$\ddot{S}Z=(\ddot{S}(i,\alpha))_{\alpha=1,...,q, i=1,...,p}.$$ 
Applying Lemma \ref{s-matrix}  to $\overrightarrow{M}$ gives 
$$\overrightarrow{W}\otimes Z s=\overrightarrow{W}\dot{s}\otimes (\ddot{S}(i,\alpha))=\overrightarrow{W}\otimes \dot{s}(\ddot{S}(i,\alpha)).$$
Since $W^1,...,W^q$ are linearly independent in $K_1$ we conclude that
$$Zs=\dot{s}(\ddot{S}(i,\alpha))$$ or equivalently
$$(\ddot{S}(i,\alpha))=(\dot{s})^{-1}Zs=\overline{\dot{s}}Zs$$
by noting that $\dot{s}, \ddot{s}, s$  are symmetric, unitary.
\begin{lem}\label{s-matrix1} For any $i\in I$ and $\alpha\in J_i$ we have
	$$\ddot{S}(i,\alpha)=\sum_{j\in I,\beta\in J_j}\overline{\dot{s}_{\alpha,\beta}}s_{i,j}(j,\beta),$$
	and every simple object in $\CC_2$ appears in some $(i,\alpha).$
	\end{lem}
\begin{proof}
	We only need to show that every simple object in $\CC_2$ appears in some $(i,\alpha).$ Note that $\ddot{S}1_{\CC_2}=\sum_{\phi\in K}\ddot{s}_{1\phi}N^{\phi}$ 
	and $\ddot{s}_{1\phi}\ne 0$ for all $\phi\in K$ as $\dim N^\phi=\frac{\ddot{s}_{1\phi}}{\ddot{s}_{11}} $ is positive. 
	So we have
	$$\ddot{S}1_{\CC_2}=\ddot{S}(1,1)=\sum_{j\in I,\beta\in J_1}\overline{\dot{s}_{1,\beta}}s_{1,j}(j,\beta).$$
Writing each $(i,\alpha)$ as a linear combination of $N^{\phi}$ we see that each $N^{\phi}$
must appears in some $(i,\alpha).$
\end{proof}

For $i\in I$ and $\alpha\in J_i,$ we set 
$$b(i,\alpha)=\sum_{\beta\in \KW}\overline{\dot{s}_{\alpha\beta}}s_{ij}$$
where  $M^j=a_{\beta\otimes 1}.$
 We have
\begin{prop}\label{p4.3} We have the dimension formula  
	$$\dim(i,\alpha)=\frac{b(i,\alpha)}{b(1,1)}$$
	for all $i\in I, \alpha\in J_i.$
	In particular, $b(i,\alpha)\ne 0.$
\end{prop}
\begin{proof} Note  from lemmas \ref{kwl1} and \ref{s-matrix1} that the coefficient of $N^1=1_{\CC_2}$ in 
	$\ddot{S}(i,\alpha)$ is $b(i,\alpha).$  In particular, $\ddot{s}_{11}=b(1,1).$ The result follows the definition of categorical dimension immediately. Since $\CC_2$ is pseudo unitary,  we see that $b(i,\alpha)\ne 0.$  
\end{proof}

The $b(i,\alpha)$ was first introduced in \cite{KW} to study the coset constructions for affine Kac-Moody algebras. It was conjectured \cite[Conjecture 2.5]{KW} that $b(i,\alpha)\ne 0$ for
$\alpha\in J_i.$ Proposition \ref{p4.3} asserts that the Kac-Wakimoto conjecture in the coset setting associated with pseudo unitary modular tensor categories is always true. 

\section{Identifications }

In this section, we prove that for any two modules $M^i, M^j\in \CC$ either the simple objects of $\CC_2$ appearing in $M^i, M^j$ are the same, or there is no intersection between simple objects of $\CC_2$ appearing in $M^i$ and $M^j.$ We also gives a dimension formula for any $(i,\alpha)$ in terms of dimensions of $M^i$ and $W^{\alpha},$ and discuss when $\KW$ forms a group.

Recall the braided fusion category $\CC^{kw}$ from Section 4. Clearly,  both $\CC$ and $(\CC_1\otimes \CC_2)_A$  are right  $\CC^{kw}$-modules.  There is  an equivalence relation on $\Or((\CC_1\otimes\CC_2)_A)$ such that  $X\equiv Y$ iff there exists
$Z\in\Or(\CC^{kw})$ such that $X$ is a direct summand of $Z\boxtimes Y$ \cite{EGNO}. This also defines an equivalence relation on $\Or(\CC).$  Note that the equivalence class of $A$ is exactly the set $\{a_{\alpha\otimes 1}|\alpha\in\KW\}.$ For $X\in (\CC_1\otimes\CC_2)_A$ 
we denote $X_{\CC_i}$ the set of simple objects of $\CC_i$ appearing in $X$ for $i=1,2.$

The following result gives the first identification between $M^i_{\CC2}$ and $M^j_{\CC_2}.$
\begin{prop}\label{identify}  Let $X,Y\in \Or(\CC).$  Then
	
	(1)  If $X,Y$ are equivalent, then $X_{\CC_2}=Y_{\CC_2}.$ 
	
	(2) If $X,Y$ are inequivalent, then $X_{\CC_2} \cap Y_{\CC_2}=\emptyset.$ 
	\end{prop}
\begin{proof}
	(1) Let  $X$ is a direct summand of $Y\boxtimes_A Z$ for some $Z\in \CC^{kw}.$ Note from  Lemma \ref{a} that $Z=a_{\gamma\otimes 1}$ for some $\gamma\in \KW.$  Let $Y=\oplus_{\alpha\in J}W^{\alpha}\otimes Y^{\alpha}$ in $\CC_1\otimes\CC_2.$ So 
	$$Y\boxtimes_A Z=Y\boxtimes_{\CC_1\otimes \CC_2} (\gamma \otimes 1_{\CC_2})=\oplus_{\alpha\in J}(\alpha\boxtimes \gamma)\otimes Y^{\alpha}.$$ 
	This implies that the simple objects of $\CC_2$ appear in $X$ are subset of simple objects of $\CC_2$ appear in $Y.$ Similarly,
	the simple objects of $\CC_2$ appear in $Y$ are subset of simple objects of $\CC_2$ appear in $X.$
	
	(2)  Let  $X=\oplus_{\alpha\in J}W^{\alpha}\otimes X^{\alpha}$ in $\CC_1\otimes\CC_2.$  Assume that  $X^\alpha\ne 0$ and $Y^\beta\ne 0.$ By Lemma \ref{kwl2} and Proposition \ref{p3.6} we have 
$$\<X^{\alpha}, Y^{\beta}\>=\<a_{1\otimes X^{\alpha}},a_{1\otimes Y^{\beta}}\>\leq 	\<X\boxtimes_A a_{\alpha'\otimes 1},Y\boxtimes_A a_{\beta'\otimes 1}\>=\sum_{\epsilon\in \KW}N_{\alpha' \beta}^{\epsilon} N_{X'Y}^{M^k}$$
where $M^k=a_{\epsilon\otimes 1}.$
Note that
 $$N_{X',Y}^{M^k}=\dim \Hom_A(X'\boxtimes_A Y, M^k)=\dim\Hom_A(Y,X\boxtimes_A M^k)=0$$
  as $X,Y$ are inequivalent.
This forces $\<X^{\alpha}, Y^{\beta}\>=0.$ Consequently, simple objects of $\CC_2$ appear in $X$ and $Y$ are inequivalent. 
\end{proof}

For later discussion, we need to consider a subcategory $\DD$ of $(\CC_1\otimes \CC_2)_A$ generated by $a_{\alpha\otimes 1}$ for all $\alpha\in J.$ It is clear that $\DD$ is a braided fusion category which is braided equivalent to category $\CC_1.$ Then $(\CC_1\otimes \CC_2)_A$ is a ${\DD}$-module category. As before there is an equivalence relation on $(\CC_1\otimes \CC_2)_A.$  Using the proof of Proposition \ref{identify}, Remarks \ref{r4.8} and \ref{r4.10} we have 

\begin{prop}\label{identify1}  Let $X,Y\in \Or((\CC_1\otimes \CC_2)_A).$  Then
	
	(1)  If $X,Y$ are equivalent, then $X_{\CC_2}=Y_{\CC_2}.$ 
	
	(2) If $X,Y$ are inequivalent, then $X_{\CC_2} \cap Y_{\CC_2}=\emptyset.$ 
\end{prop}

\begin{cor}\label{key} Let $X,Y\in \Or((\CC_1\otimes \CC_2)_A).$ Then either $X_{\CC_1}=Y_{\CC_1}$ or $X_{\CC_1} \cap Y_{\CC_1}=\emptyset.$  
\end{cor}
\begin{proof} From Proposition \ref{identify1} we know that either the simple objects of $\CC_2$ in $X$ and in $Y$ are exactly the same, or there is no intersection between the sets of simple objects of $\CC_2$ appearing in $X$ and $Y.$ It is clear that the same conclusion holds if we replace $\CC_2$ by $\CC_1.$
	\end{proof}

\begin{lem}\label{al5.5}  Let $i\in I$ and $\alpha\in J_i.$ Then 
	$$a_{1\otimes (i,\alpha)}=\sum_{\gamma\in J_i, \beta\in J_1}N_{\gamma \alpha'}^\beta \beta\otimes (i,\gamma)=\sum_{\beta\in J_1}{\beta}\otimes(\sum_{\gamma\in J_i}N_{\gamma \alpha'}^\beta(i,\gamma))=\sum_{\beta\in J_1}{\beta}\otimes(\sum_{\gamma\in J_i}N_{\beta \alpha}^\gamma(i,\gamma))$$
	is a direct summand of $M^i\boxtimes_Aa_{\alpha'\otimes 1}$ and $(1,\beta)\boxtimes (i,\alpha)=\sum_{\gamma\in J_i}N_{\beta \alpha}^\gamma(i,\gamma)$ for $\beta\in J_1.$ 
\end{lem}
\begin{proof} From Lemma \ref{kwl2} we know that $a_{1\otimes (i,\alpha)}$ is a direct summand of $M^i\boxtimes_Aa_{\alpha'\otimes 1}.$ Note that 
$$M^i\boxtimes_A a_{\alpha'\otimes 1}=M^i\boxtimes_{\CC_1\otimes\CC_2}(\alpha'\otimes 1)=\sum_{\gamma\in J_i, \beta\in J_1}N_{\gamma \alpha'}^\beta \beta\otimes (i,\gamma)\oplus\sum_{\gamma\in J_i, \beta\notin J_1}N_{\gamma \alpha'}^\beta \beta\otimes (i,\gamma).$$
Set 
		$$X=\sum_{\gamma\in J_i, \beta\in J_1}N_{\gamma \alpha'}^\beta \beta\otimes (i,\gamma).$$
		Then there are nonnegative integers $n_{\beta}$ and $U_{\beta}\in \CC_2$ for $1\ne \beta\in J_1$ such that 
		$$X=1_{\CC_1}\otimes (i,\alpha)+\sum_{1\ne \beta\in J_1}n_{\beta}\beta\otimes U_{\beta}.$$
Since $\cal  F^1$ is a braided fusion category by Theorem \ref{lin}, we see that $X\in (\CC_1\otimes \CC_2)_A$ is an $A$-module which is a direct sum of simple $A$-modules. Let  $(i,\alpha)=\oplus_{s=1}^t Z_s$ be a direct sum of simple objects in $\CC_2.$ It follows from Lemma \ref{a} that $a_{1\otimes Z_s}$ is a simple $A$-module.
We claim that $X=\oplus_{s=1}^ta_{1\otimes Z_s}.$ Clearly, $\oplus_{s=1}^ta_{1\otimes Z_s}$ is an $A$-submodule of $X.$ If $X\ne \oplus_{s=1}^ta_{1\otimes Z_s}$  we take an arbitrary simple $A$-submodule $N$ of $X$ such that 
$N\cap \oplus_{s=1}^ta_{1\otimes Z_s}=\{0\}.$ Clearly, $A_{\CC_1}\cap N_{\CC_1}\ne \emptyset.$ By Corollary \ref{key}, $A_{\CC_1}= N_{\CC_1}.$ But from the discussion above,  we see that $1_{\CC_1}\in A_{\CC_1}$ and $1_{\CC_1}\notin N_{\CC_1}.$  This is a contradiction. On the other hand, $a_{1\otimes(i,\alpha)}=a_{1\otimes (\oplus_{s=1}^t Z_i)}=\oplus_{s=1}^ta_{1\otimes Z_s}.$
  Thus 
   $X=a_{1\otimes (i,\alpha)},$ as desired.

The equality $(1,\beta)\boxtimes (i,\alpha)=\sum_{\gamma\in J_i}N_{\gamma \alpha'}^\beta(i,\gamma)$ for $\beta\in J_1,$  follows from the comparison 
$$a_{1\otimes(i,\alpha)}
=\sum_{\beta\in J_1}\beta\otimes(\sum_{\gamma\in J_i}N_{\gamma \alpha'}^\beta(i,\gamma))$$
with 
$$a_{1\otimes(i,\alpha)}=A\boxtimes_{\CC_1 \otimes \CC_2}  (1_{\CC_1}\otimes (i,\alpha))=
\oplus_{\beta\in J_1} {\beta}\otimes ((1,\beta)\otimes {(i,\alpha)}).$$
So it suffices to show that $N_{\gamma \alpha'}^\beta=N_{\beta\alpha}^\gamma$ for any $\beta,\alpha,\gamma\in \CC_1.$ Since 
$$\beta\boxtimes \alpha=\sum_{\gamma\in\CC_1}N_{\beta\alpha}^\gamma\gamma$$
and
$$\sum_{\gamma\in\CC_1}N_{\beta'\alpha'}^\gamma\gamma=\beta'\boxtimes \alpha'=\alpha'\boxtimes \beta'=
(\beta\boxtimes \alpha)'=\sum_{\gamma\in\CC_1}N_{\beta\alpha}^{\gamma}\gamma'$$
we have $N_{\beta\alpha}^{\gamma}=N_{\beta'\alpha'}^{\gamma'}=N_{\gamma\alpha'}^{\beta}.$ \end{proof}

Recall Proposition \ref{p4.3}. Now we can give another  formula of $d_{(i,\alpha)}$ without using $S$-matrices for $i\in I$ and $\alpha\in J_i$
\begin{thm}\label{mtheorem1} Let $i\in I.$  Then $d_{(i,\alpha)}=c_i d_id_{\alpha}$ for all $\alpha\in J_i$ where $c_i=\frac{\sum_{\beta\in J_1}d_{\beta}^2}{\sum_{\gamma\in J_i}d_{\gamma}^2}\leq 1.$
		\end{thm}
\begin{proof} From Lemma \ref{al5.5} we have identity
	$$d_{(1,\beta)}d_{(i,\alpha)}=\sum_{\gamma\in J_i}N_{\beta\alpha}^{\gamma}d_{(i,\gamma)}.$$
	Using Theorem  \ref{lin} gives $d_\beta=d_{(1,\beta)}.$ So we have identity
		$$d_\beta d_{(i,\alpha)}=\sum_{\gamma\in J_i}N_{\beta\alpha}^{\gamma}d_{(i,\gamma)}.$$
		Let $N_\beta$ be a $|J_i|\times |J_i|$ matrix such that $(N_\beta)_{\alpha \gamma}=N_{\beta,\alpha}^{\gamma}.$ Let ${\bf d}$ be a vector in
		$\R^{|J_i|}$ with components $d_{(i,\alpha)}.$ Then 
		$$N_\beta {\bf d}=d_\beta {\bf d}$$
		and ${\bf d}$ is a common eigenvectors for all $N_{\beta}$ with eigenvalues $d_{\beta}.$
		
		On the other hand,  we also have 
			$$d_\beta d_\alpha=\sum_{\gamma\in J_i}N_{\beta\alpha}^{\gamma}d_\gamma,$$
			$$N_\beta{\bf e}=d_{\beta}{\bf e}$$
for all $\beta$ where ${\bf e}\in \R^{|J_i|}$ with components $d_{\alpha}.$ 

Set $N=\sum_{\beta\in J_1}N_\beta.$ We claim that $N_{\alpha \gamma}=\sum_{\beta\in J_1}(N_{\beta})_{\alpha \gamma}>0$ for all $\alpha, \gamma\in J_i.$ That is, $N$ is a strictly positive matrix. Equivalently, for any $\alpha, \gamma\in J_i$ there exists $\beta\in J_1$ such that $(N_{\beta})_{\alpha \gamma}=N_{\beta\alpha}^{\gamma}>0.$ This essentially follows from the fact that $M^i$ is a simple $A$-module. Consider
$A$-module 
$$A\boxtimes (\alpha\otimes (i,\alpha))=\oplus_{\beta\in J_1}(\beta\boxtimes \alpha)\otimes ((1,\beta)\boxtimes (i,\alpha)).$$
Since  $\Hom_A(A\boxtimes (\alpha\otimes (i,\alpha)), M^i)=\Hom_{\CC_1\otimes\CC_2}( \alpha\otimes (i,\alpha), M^i)$ is nonzero, $M^i$ is a summand of $A\boxtimes (\alpha\otimes (i,\alpha)).$ So there exists $\beta\in J_1$ such that $N_{\beta\alpha}^{\gamma}$ is nonzero.

Since the entries of $N$ are positive and both ${\bf d}$ and ${\bf e}$ are eigenvectors of $N$ with positive components and with eigenvalue
$\sum_{\beta\in J_1}d_{\beta},$ from the well-known Frobenius-Perron Theorem 
	that there exists a positive number $c_i$ such that ${\bf d}=c_id_i{\bf e}.$ That is, $d_{(i,\alpha)}=c_id_id_{\alpha}$ for all $\alpha\in J_i.$ Note that 
	$$d_i=\frac{\sum_{\alpha\in J_i}d_{\alpha}d_{(i,\alpha)}}{\sum_{\beta\in J_1}d_{\beta}d_{(1,\beta)}}=\frac{\sum_{\alpha\in J_i}c_id_id_{\alpha}^2}{\sum_{\beta\in J_1}d_{\beta}^2}.$$
So $c_i=\frac{\sum_{\beta\in J_1}d_{\beta}^2}{\sum_{\gamma\in J_i}d_{\gamma}^2}.$ By Lemma \ref{kwl2}, $a_{1\otimes (i,\alpha)}$ is a submodule of $M^i\boxtimes a_{\alpha'\otimes 1}.$ Thus $\d_{(i,\alpha)}\leq \d_i\d_\alpha$ and  $c_i\leq 1,$ proof is complete.
	\end{proof}

\begin{thm}\label{KWd} The following are equivalent:
	
	(1)  $\KW$  forms a group,
	
	(2) $\d_{(i,\alpha)}=\d_i\d_{\alpha}$ for 
	all $i\in I, \alpha\in J_i,$

(3)  For any $i\in I,$ $c_i=\frac{\sum_{\beta\in J_1}d_{\beta}^2}{\sum_{\gamma\in J_i}d_{\gamma}^2}=1,$

(4) For any $\alpha\in J,$ $W^{\alpha'}\boxtimes W^{\alpha}=\sum_{\gamma\in J_1}N_{\alpha,\alpha'}^{\gamma}W^{\gamma}.$ 
\end{thm}

\begin{proof}
	
	First, by Theorem \ref{mtheorem1}, (2) and (3) are equivalent.
	
	(2) $\Rightarrow$ (1): We first assume that $\d_{(i,\alpha)}=\d_i\d_{\alpha}$ for 
	all $i\in I, \alpha\in J_i.$ Recall from Proposition \ref{p4.3} that  $\d_{(i,\alpha)}=\frac{b{(i,\alpha)}}{b{(1,1)}}$,
	where $b{(i,\alpha)}=\sum_{\beta\in \KW} s_{ij}\overline{\dot{s}_{\alpha\beta}}$, 
	$b{(1,1)}=\sum_{\beta\in \KW} s_{1j}\overline{\dot{s}_{1\beta}}.$
	Note that $s_{1i}, \dot{s}_{1\beta}\in \BR_{+}$, and  $|s_{ij}|\leqslant \d_i s_{1j}$, 
	$|\dot{s}_{\alpha \beta}|\leqslant \d_{\alpha}\dot{s}_{1\beta}$.
	Then $$\d_{(i,\alpha)}\leqslant\frac{\sum_{\beta\in \KW} \d_i s_{1j}{\d_{\alpha}\dot{s}_{1\beta}}}{\sum_{\beta\in \KW} s_{1j}{\dot{s}_{1\beta}}}=\d_i \d_{\alpha}.$$
	Thus $|s_{ij}|=\d_i s_{1j}$.
	Moreover, $|s_{ij}|=\d_i s_{1j}=\d_j s_{1i}$, and $|\dot{s}_{\alpha \beta}|= \d_{\alpha}{\dot{s}}_{1\beta}=\d_{\beta}{\dot{s}}_{1\alpha}.$
	
	For $\beta\in \KW,$ we have $1=\sum_{\alpha} |s_{\alpha\beta}|^2=\d_\beta^2\sum_{\alpha} s_{1\alpha}^2=\d_\beta^2$ as $\dot s$ is unitary.  So $\d_\beta=1$ for all $\beta\in \KW$ and $\KW$  forms an abelian group. 
	
	(1) $\Rightarrow$ (2): We now assume that $\KW$ forms a group, equivalently, $\d_\beta=1$ for all $\beta\in \KW.$ So $\d_j=d_\beta=1$ where
 $M^j=a_{\beta\otimes1}.$
	Let $i\in I$ and $\alpha\in J_i.$ We claim that $s_{ij}\overline{\dot{s}_{\alpha\beta}}$ is positive.
	
	Let $k\in I,$ $\gamma\in J$ such that $M^i\boxtimes M^j\cong M^k,$ $W^\alpha\boxtimes W^\beta\cong W^\gamma.$ Then  
	$$s_{ij}=\frac{1}{D}\frac{\theta_k}{\theta_i\theta_j}\d_k=\frac{1}{D}\frac{\theta_k}{\theta_i\theta_j}\d_i$$
	where $\theta_i=\theta_{M^i},$ 
	and 
	$$\dot{s}_{\alpha\beta}=\frac{1}{D_1}\frac{\theta^1_\gamma}{\theta^1_\alpha\theta^1_\beta}\d_\alpha$$
	where $\theta^1_\alpha=\theta^1_{W^\alpha}.$  By \cite[Theorem 1.17]{KO}, we know that
	$\theta_i=(\theta^1\otimes \theta^2)_{M^i}$ is a constant. So $\theta_i=\theta^1_{\alpha}\theta^2_{(i,\alpha)},$
	$\theta_j=\theta^1_{\beta}\theta^2_{(j,\beta)}$ and $\theta_k=\theta^1_{\gamma}\theta^2_{(k,\gamma)}.$ 
 In particular, $\theta^2_{(i,\alpha)}$ is a constant.  Note that $M^j=a_{\beta\otimes 1}=\oplus_{\delta\in J_1}(\delta\boxtimes \beta)\otimes (1,\delta).$ Since $\beta$ is a simple current we see that $\delta\boxtimes \beta=\beta$ if and only if $\delta=1_{\CC_1}.$ This implies that $(j,\beta)=1_{\CC_2}$ and  
 $\theta_j=\theta^1_{\beta}.$ Furthermore, the relation $M^k=M^i\boxtimes (\beta\otimes 1)$ gives $(k,\gamma)=(i,\alpha).$  So  $\theta_k=\theta^1_{\gamma}\theta^2_{(i,\alpha)}.$  As $\theta$'s are roots of unity, we see that $s_{ij}\overline{\dot{s}_{\alpha\beta}}$ is positive.
	
	The  relation  $1=\sum_{i} |s_{ij}|^2\leq \d_j^2\sum_{i} s_{1i}^2=\d_j^2=1$ forces 
	$|s_{ij}|=s_{i1}\d_j=s_{j1}\d_i$ for all $i.$ Similarly, $|\dot{s}_{\alpha,\beta}|=\dot{s}_{\alpha1}\d_{\beta}=\dot{s}_{\beta1}\d_{\alpha}$ for
	all $\alpha\in J_i.$ Thus 
	$$\d_{(i,\alpha)}=\frac{b(i,\alpha)}{b(1,1)}=\frac{\sum_{\beta\in \KW}\overline{\dot{s}_{\alpha\beta}}s_{ij}}{\sum_{\beta\in \KW}\overline{\dot{s}_{1\beta}}s_{1j}}=\frac{\sum_{\beta\in \KW}{\dot{s}_{1\beta}}s_{j1}\d_i\d_{\alpha}}{\sum_{\beta\in \KW}{\dot{s}_{1\beta}}s_{j1}}=\d_i\d_\alpha,$$ 
	as expected. 
	
	(2) $\Rightarrow$ (4):  By Lemma \ref{s-matrix1}, there exists $i\in I$ such that $\alpha\in J_i.$  Note from Lemma \ref{kwl2} that
	$$a_{1\otimes (i,\alpha)}=M^i\boxtimes a_{\alpha'\otimes 1}=\sum_{\gamma\in J_i}({\gamma}\boxtimes {\alpha'})\otimes (i,\gamma).$$
	From the definition of $a_{1\otimes (i,\alpha)}$ we know that 
	$$a_{1\otimes (i,\alpha)}=\sum_{\beta\in J_1}\beta\otimes ((1,\beta)\boxtimes (i,\alpha)).$$
	(4) follows.
	
	(4) $\Rightarrow$ (2): Let $i\in I$ and $\alpha\in J_i.$ 
 Note that $M^i$ is a direct summand of $a_{\alpha\otimes (i,\alpha)}$ and $M^i\boxtimes a_{\alpha'\otimes 1}$ is a direct summand of $a_{\alpha\otimes (i,\alpha)}\boxtimes _A a_{\alpha'\otimes 1}\cong a_{(\alpha\boxtimes \alpha')\otimes (i,\alpha)}.$ Using the assumption that 
	${\alpha}\boxtimes {\alpha'}=\sum_{\beta\in J_1}N_{\alpha\alpha'}^\beta {\beta}$ asserts that if
 $N$ is an $A$-submodule of $a_{(\alpha\boxtimes \alpha')\otimes (i,\alpha)},$  then $N_{\CC_1}=J_1$ by Corollary \ref{key}. It is evident from $M^i\boxtimes a_{\alpha'\otimes 1}=\bigoplus_{\gamma\in J_i}(\gamma\boxtimes \alpha')\otimes (i,\gamma)$ that the muliplicity space  of $1_{\CC_1}$ in $M^i\boxtimes a_{\alpha'\otimes 1}$ is exactly  ${(i,\alpha)}.$  From the proof of Lemma \ref{al5.5} we conclude that 
  $a_{1\otimes (i,\alpha)}=M^i\boxtimes a_{\alpha'\otimes 1},$ which is equivalent to $\d_{(i,\alpha)}=\d_i\d_{\alpha}.$\end{proof}  
	
 We note that the KW set forms a group in the diagonal cosets considered in \cite{X1}.

\section{Kac-Wakimoto Hypothesis}
	
	 Kac-Wakimoto Hypothesis  was proposed   in \cite{KW}:  For $i\in I, \alpha\in J_i, $ $\beta\in \KW $ and  $M^j= a_{\beta\otimes 1}$ then  $s_{ij}\overline{\dot{s}_{\alpha\beta}}\geq 0.$ We prove Kac-Wakimoto Hypothesis  in this section.
	 
	 First, we need some general results that will be used in the proof of the Kac-Wakimoto Hypothesis.  Consider a modular tensor category ${\DD}$ and a regular commutative algebra $B\in \DD.$ Set $a_{\lambda}=B\boxtimes_{\DD}\lambda$ for  $\lambda\in \Or(\DD).$   Also  set $V=\Or(\DD_B)$ and 
	  $\Or(\DD_B^{0})=\{\sigma_i|i\in\Delta\}.$ For each $M\in \DD_B$ define a linear operator 
	 $$T_M: K( \DD_B) \to K(\DD_B) $$
	 such that $T_M(N)=M\boxtimes_BN$ for $N\in \Or(\DD_B).$ For short we denote $T_{\sigma_i}$ by $T_i$ and 
	 $T_{a_{\lambda}}$by $T_{\lambda}.$ We claim that linear operators 
	 $$\{T_{i}, T_{\lambda} | i\in \Delta, \lambda\in \Or(\DD)\}$$
	 commute with each other.  Clearly, $T_{i}, T_{j}$ commute for any $i,j.$ From Section 3.3 of \cite{DMNO} we know that 
	 $a_{\lambda}\boxtimes_B N$ and $N\boxtimes_B a_{\lambda}$ are isomorphic $B$-modules. So $T_{\lambda}$ commutes with any $T_M.$  It is well known that $\lambda\mapsto a_{\lambda}$ is an algebra homomorphism from $K(\DD)$ to $K(\DD_B).$

	 Note that  $K( \DD_B) $ is a module for both  $K(\DD_B^0)$ and $K(\DD).$ The action of $a\in \Or(\DD)$ is given by $T_a.$ It is well known that   $K(\DD_B^0)$ and $K(\DD)$ are semisimple commutative algebras and the 
	 the irreducible representations of $K(\DD_B^0)$ and $
	 K(\DD)$ are given by characters associated to $i\in\Delta$ and $\mu\in \Or(\DD)$ 
	 $$\sigma^j\mapsto \frac{s_{ji}}{s_{1i}} $$
	 $$\lambda \mapsto \frac{{s}^{\DD}_{\lambda\mu}}{{s}^{\DD}_{1\mu}} $$
	 where $s=(s_{ij})$ is the normalized  $S$-matrix associated to $\DD_B^0$ and $s^{\DD}=(s_{\lambda\mu}^\DD)$ is the normalized  $S$-matrix associated to $\DD.$
	 
	 Define a positive definite hermitian form $(,)$ on  $K(\DD_B)$ so that $a\in D$ forms an orthonormal basis.  
	 Since  the operators $T_{j}, T_{\lambda}$ for $j\in \Delta$ and $\lambda\in \Or(\DD)$ can be diagonalized simultaneously,
	  there exists an orthonormal basis $v^{(i,\mu,m)}$ with $i\in \Delta$ and $\mu\in D$ such that
	 $$T_{j}v^{(i,\mu,m)}=\frac{s_{ji}}{s_{1i}}v^{(i,\mu,m)} $$
	 $$T_{\lambda}v^{(i,\mu,m)}= \frac{{s}^\DD_{\lambda\mu}}{{s}^\DD_{1\mu}} v^{(i,\mu,m)}$$
	 for all $j$ and $\lambda$ where $m$ is the index of basis vectors of the eigenspace with indicated eigenvalues. It is possible that for some $i,\mu,$ there are no eigenvectors $v^{(i,\mu,m)}.$ Let  $E$ be the  set of $(i,\mu,m)$ such that $v^{(i,\mu,m)}$ exists.

	 For $\lambda\in \Or(D)$ and $a\in V$ 
	 $$T_\lambda(a)
	 = a_{\lambda}\boxtimes_B a=\sum_{b\in V} V^\lambda_{ab} b
	 $$ 
	 in $K(\DD_B)$ where $V^\lambda_{ab}$ are nonnegative
	 integers.  Denote by $V^\lambda$ the matrix such that $(V^\lambda)_a^b = 
	 V^\lambda_{ab}$.  Then $V^{\mu_1} V^{\mu_2} = 	 \sum_{\mu_3} N_{\mu_1\mu_2}^{\mu_3} V^{\mu_3}.$ Note that $a_{\lambda}=\sum_{c\in V}V_{1c}^{\lambda}c.$   So we have
	 $V^\lambda = \sum_{c\in V} V^\lambda_{1c} N_c$ where $N_c$ is a matrix defined by $N_{ca}^b = \langle c\boxtimes_Ba,b\rangle$ for $a,b\in V$.
  
  Let 
	 $$v^{(i,\mu,m)}=\sum_{a\in V}v^{(i,\mu,m)}_aa.$$
  Recall a well known result from linear algebra on a diagonalizable matrix $X=(x_{ij})_{i,j=1,...,n}$ over $\C.$ Let ${\bf v}_i=\left(\begin{array}{c} v_{1i}\\ \vdots\\ v_{ni}\end{array}\right)$ be the eigenvector of $X$ with eigenvalue $t_i$ for $i=1,...,n$ such that these eigenvectors form an orthonormal basis of $\C^n.$ Then $x_{ij}=\sum_{k=1}^nt_kv_{ik}v_{jk}^*$ where $v_{jk}^*$ is the complex conjugate of $v_{jk}.$
  Applying this result to matrix $V^{\lambda}$ yields
	 $$V^\lambda_{ab} = \sum_{(i,\mu,m)\in E}
	 \frac{s^\DD_{\lambda \mu}}{s^\DD_{1\mu}} \cdot v_a^{(i,\mu,m)}
	 v_b^{(i,\mu,m)^*}$$

	The following result was given in \cite{X2, X3} in the setting of conformal nets. 
	 \begin{thm}\label{KWH7.1}
	 $(i,\mu,m)\in E$   if and only if $\sigma_i$ is a direct summand of $a_{\mu}.$   
	 
	 \end{thm}
	 \begin{proof} The proof is divided into several steps.
	 	
	 	(1) For $j\in \Delta, \lambda\in \Or(\DD)$ we set $b_{j\lambda}=\<\sigma_j,a_{\lambda}\>=\dim\Hom_B(\sigma_j,a_{\lambda}).$  We claim that 
	 	$$
	 	\sum_j s_{kj} b_{j\lambda} = \sum_{\mu} b_{k\mu} s^\DD_{\mu \lambda}.
	 	$$
	Let $F: K(\DD)\to K(\DD_B)$ be the algebra homomorphism induced by the functor $F: \DD\to \DD_B$ so that $F(\lambda)=a_{\lambda}.$ Note that $K(\DD_B^0)$ is subalgebra of $K(\DD_B)$ and define 
	a projection 
	$$P: K(\DD_B)\to K(\DD_B^0)$$
	 such that $P(\sigma_j)=\sigma_j$ and $P(a)=0$ if $a\in \Or(\DD_B)\setminus 	\Or(\DD_B^0).$ Let $F^0=P\circ F: K(\DD)\to K(\DD_B^0)$ be the composition. 
	It follows from   \cite[Theorem 4.1]{KO} that $F^0s^\DD=sF^0.$ Since $F^0(\lambda)=\sum_{j}b_{j\lambda}\sigma_j,$ we see that
	$$sF^0(\lambda)=\sum_k\sum_j s_{kj} b_{j\lambda}\sigma_k .$$
	 On the other hand, 
	$$F^0s^\DD(\lambda)=\sum_{\mu}s_{\mu,\lambda}^\DD P(a_{\mu})=\sum_k\sum_{\mu}s_{\mu,\lambda}^\DD b_{k\mu}\sigma_k.$$
	 	Comparing the coefficients of $\sigma_k$ in $sF^0(\lambda)$ and $F^0s^\DD(\lambda)$ gives the desired identity.

(2) For $i,j\in\Delta$ we have 	 	
 $$
v_{\sigma_j}^{(i,\mu,m)} = \frac{s_{ j' i}}{s_{1i}}
v_1^{(i,\mu,m)}.$$
To see this, consider equation
$$\sum_{a\in V}v_a^{(i,\mu,m)}\sigma_{j'}\boxtimes a=\frac{s_{ j' i}}{s_{1i}}\sum_{a\in V}v_a^{(i,\mu,m)}a.$$
Observe that $\<1,\sigma_{j'}\boxtimes a\>\ne 0$ if and only if $a=\sigma_j.$ In this case, $\<1,\sigma_{j'}\boxtimes\sigma_j\>=1.$ Comparing the coefficients of $1$  in the equation shows that  $
v_{\sigma_j}^{(i,\mu,m)} = \frac{s_{ j' i}}{s_{1i}}
v_1^{(i,\mu,m)}.$
	 	
(3) 	 Since $b_{j\lambda} = V^\lambda_{1 \sigma_j}$,  we use (2) to produce
	 $$
	 b_{j\lambda} = \sum_{(i,\mu,m)\in E} \frac{s^\DD_{\lambda \mu}}{s^\DD_{1\mu}}
	 \frac{s_{j' i}}{s_{1i}} |v_1^{(i,\mu,m)}|^2.
	 $$
	 According to the fact that $s_{j'i}=\overline{s_{ji}}$ and $s$ is unitary, we see that for fixed $k,$
	 $$
	 \sum_j s_{kj} b_{j\lambda} = \sum_j\sum_{(i,\mu,m)\in E} \frac{s^\DD_{\lambda \mu}}{s^\DD_{1\mu}} \frac{s_{kj}s_{j' i}}{s_{1i}} |v_1^{(i,\mu,m)}|^2=\sum_{(\mu,m)\in E(k)} \frac{s^\DD_{\lambda \mu}}{s^\DD_{1\mu}}
	 \frac{1}{s_{1k}} |v_1^{(k,\mu,m)}|^2
	 $$
	 where $E(k)$ is  the subset of $E$ consisting of $(k,\mu,m).$ By (1) we obtain 
	 $$
	 \sum_{\mu} b_{k\mu} s^\DD_{ \lambda\mu}=\sum_{(\mu,m)\in E(k)} \frac{s^\DD_{\lambda \mu}}{s^\DD_{1\mu}}
	 \frac{1}{s_{1k}} |v_1^{(k,\mu,m)}|^2.
	 $$
	 Multiplying this equation by $s^\DD_{\lambda'\delta} $ and sum over $\lambda\in \Or(\DD),$  and using the fact that $s^{\DD}$ is unitary, we get:
	 $$
	 \sum_{m\in E(k,\delta)} \frac{1}{s^\DD_{1\delta}s_{1k}}|v_1^{(k,\delta,m)}|^2 = b_{k\delta} , 
	 $$
	 where $m\in E(k,\delta)$ means that  $k, \delta$ are fixed and
	 $(k,\delta,m) \in E.$
	 It follows immediately that if $ b_{k\delta} >0$, then
	 $(k,\delta,m)\in E$ for some $m$. Thus we have proved that if $\sigma_k$ is a direct summand of $a_{\delta}$ then $(k,\delta,m)\in E$ for some $m.$
	 
	 (4) Now we assume that  $E(k,\delta)=\{(k,\delta, 1),..., (k,\delta, p)\}$ with
	 $(k, \delta)$ fixed.  Recall from   \cite[Corollary 3.30]{DMNO} that any $\sigma_k$ lies in the center of $K(\DD_B).$ We also know that $a_{\lambda}$ lies in the center of $K(\DD_B)$ for any $\lambda\in \DD.$
  Thus for any $a\in V,$ $T_a$ preserves the subspace of $K(\DD_B)$ spanned by vectors 
	 $v^{(k,\delta,1)},..., v^{(k,\delta, p)}$. 
	 Note that for any $a\in V$, we have:
	 \begin{eqnarray*}
	 v_a^{(k,\delta,m)} & = \langle v^{(k,\delta,m)}, a \rangle 
	 = \langle T_{a'}v^{(k,\delta,m)}, 1 \rangle = \sum_t  N_{a' s}^t v_1^{(k,\delta,t)}
	 \end{eqnarray*}
	  where $ N_{a' m}^t = \langle T_{a'}v^{(k,\delta,m)},
	 v^{(k,\delta,t)}\rangle.$ It follows that if $(k,\delta,m)\in E$, then
	 $v_a^{(k,\delta,m)} \neq 0$ for some $a\in V.$ This  implies
	 $v_1^{(k,\delta,t)} \neq 0$ for some $t$.  Using the relation $
	 \sum_{m\in E(k,\delta)} \frac{1}{s^\DD_{1\delta}s_{1k}}|v_1^{(k,\delta,m)}|^2 = b_{k\delta} 
	 $ shows that   $b_{k\delta} > 0$. 
	 \end{proof}

	 We are now in a position to prove the Kac-Wakimoto hypothesis. 

\begin{thm}\label{KWH7.2} The  Kac-Wakimoto Hypothesis is true.
\end{thm}

\begin{proof}
	
We are working in the setting of Theorem \ref{KWH7.1}.   Let $\DD=\CC_1\otimes\CC_2$ and $B=A.$ $\Delta=I.$ Then $s^{\DD}=\dot{s}\otimes\ddot{s}.$ In this case $\sigma_i=M^i$ for $i\in I.$ Let $\alpha\in J_i,$ and $x$ be a simple direct summand of $(i,\alpha).$ Set $\mu=\alpha\otimes x.$ By Theorem \ref{KWH7.1}. $(i,\mu,m)\in E$ as $\<\sigma_i, a_{\alpha\otimes x}\>=\<\sigma_i,\alpha\otimes x\>\ne 0.$ Let $\beta\in\KW$ and $M^j=a_{\beta\otimes 1}=a_{\lambda}.$ So we have 
$$T_jv^{(i,\mu.m)}=\frac{s_{ji}}{s_{1i}}v^{(i,\mu.m)},\  \  T_{\lambda}v^{(i,\mu.m)}=\frac{s^\DD_{\lambda\mu}}{s^\DD_{1\mu}}v^{(i,\mu.m)}.$$
Since $T_j=T_\lambda$ we have $\frac{s_{ji}}{s_{1i}}=\frac{s^\DD_{\lambda\mu}}{s^\DD_{1\mu}}.$ It is easy to see that 
$s^\DD_{\lambda\mu}=\dot{s}_{\beta\alpha}\ddot{s}_{1x}$ and $s^\DD_{1\mu}=\dot{s}_{1\alpha}\ddot{s}_{1x}.$ Since  
$\ddot{s}_{1x}\ne 0$ we immediately have $\frac{s^\DD_{\lambda\mu}}{s^\DD_{1\mu}}=\frac{\dot{s}_{\beta\alpha}}{\dot{s}_{\alpha1}}=\frac{s_{ji}}{s_{1i}}.$  Using the fact that 
$s_{1i}>0$ and $\dot{s}_{\alpha1}>0$ we conclude that $s_{ji}\overline{\dot{s}_{\beta\alpha}}\geq 0.$ That is, the Kac-Wakimoto Hypothesis is true in the categorical coset construction. \end{proof}

\section{More results on $(i,\alpha)$}

In  this section we assume that $\d_{(i,\alpha)}=\d_i\d_{\alpha}$ for $i\in I$ and $\alpha\in J_i,$ and discuss how to decompose $(i,\alpha)$ into a direct sum of simple objects in $\CC_2.$

By Theorem \ref{KWd}, $G=\Or(\CC_1^{kw})$ forms an abelian group.  For any $i\in I,$ $\alpha\in J_i,$  let 
 $$G^i=\{\beta\in\CC_1^{kw}|a_{\beta\otimes 1}\boxtimes_AM^i\cong M^i\},$$
$$G^{(i,\alpha)}=\{\beta\in G^i| \beta \boxtimes \alpha\cong \alpha\}.$$
Clearly,  $G^{(i,\alpha)}$ is a subgroup of $G^i$ which is a subgroup of $G.$ 
Notice that   the inverse of $a_{\beta\otimes 1}$ is $a_{\beta'\otimes 1}.$

\begin{lem} If $\beta\in G^{(i,\alpha)}$ then $\beta\in J_1$ and $a_{\beta\otimes 1}=a_{1\otimes (1,\beta')}.$ In particular, $(1,\beta')\boxtimes (i,\alpha)\cong (1,\beta)\boxtimes (i,\alpha)\cong (i,\alpha).$
\end{lem}
 \begin{proof} 
  By Lemma \ref{kwl2} we see that
  $$a_{\beta\otimes (i,\alpha)}=a_{\beta\otimes 1}\boxtimes a_{1\otimes (i,\alpha)}=a_{\beta\otimes 1}\boxtimes M^i\boxtimes a_{\alpha'\otimes1}=M^i\boxtimes a_{\alpha'\otimes1}={1\otimes (i,\alpha)}.$$ 
  This implies that $\beta\in J_1,$  $a_{\beta\otimes 1}=a_{1\otimes (1,\beta')}$ and $(1,\beta')\boxtimes (i,\alpha)\cong (i,\alpha).$  The isomorphism $(1,\beta)\boxtimes (i,\alpha)\cong (i,\alpha)$ is clear. \end{proof}
 
 	

 	
 Notice that  $G^i$ acts on $\{W^{\alpha}|\alpha\in J_i\}.$  We have 
 \begin{thm}\label{KWd1}  Let $\alpha,\gamma\in J_i.$  Then 
	
(1) All simple summands of $(i,\alpha)$ have the same dimension. 

(2) $(i,\alpha)\cong (i,\gamma)$ if and only if $W^\gamma$ and $W^\alpha$ are in the same $G^i$-orbit.

(3) $\<(i,\alpha), (i,\beta)\>=0$ if and only if $W^\gamma$ and $W^\alpha$ are not in the same $G^i$-orbit.
\end{thm}
\begin{proof}
(1) Set $B=\oplus_{\beta\in G^{(i,\alpha)}}\beta\otimes (1,\beta).$ Then $B$ is a regular commutative algebra in
$\CC_1\otimes\CC_2$ and $\alpha\otimes (i,\alpha)\in (\CC_1\otimes\CC_2)_B^0.$ 
We first show that $\alpha\otimes (i,\alpha)$ is simple $B$-module.
Let $N=\alpha\otimes X$ be a simple direct submmand of  $\alpha\otimes (i,\alpha).$
We denote the Kac-Wakimoto set associated to $B$  by $\KW^B.$ Then $\KW^B=G^{(i,\alpha)}.$
For any $\beta\in G^{(i,\alpha)},$ $a_{\beta\otimes  1}^B=B\boxtimes_{\CC_1\otimes  \CC_2} (\beta\otimes 1)$
and $a^B_{\beta\otimes 1}\boxtimes_BN\cong N.$ 

Applying  Lemma \ref{al5.5} with $M^i,  (i,\alpha)$ replacing by $N, X$ respectively, 
we know that 
$$a_{1\otimes X}^B=B\boxtimes_{\CC_1\otimes\CC_2}(1\otimes X)=\sum_{\beta\in G^{(i,\alpha)}}N_{\alpha \alpha'}^\beta \beta\otimes X=\sum_{\beta\in G^{(i,\alpha)}} \beta\otimes X$$
is a direct summand of
\begin{eqnarray*}
	N\boxtimes_Ba^B_{\alpha'\otimes 1}&=&N\boxtimes_{\CC_1\otimes\CC_2}(\alpha'\otimes 1)\\
	&=&(\alpha\boxtimes \alpha')\otimes X\\
	&=&\sum_{\beta\in G^{(i,\alpha)}}\beta\otimes X+\sum_{\beta\notin G^{(i,\alpha)}}N_{\alpha,\alpha'}^\beta \beta\otimes X.
\end{eqnarray*}
Thus 
\begin{eqnarray*}
a_{1\otimes X}^B\boxtimes_Ba_{\alpha\otimes 1}^B&=&\sum_{\beta\in G^{(i,\alpha)}} (\beta\boxtimes \alpha)\otimes X\\
&\cong& \sum_{\beta\in G^{(i,\alpha)}} (\alpha\boxtimes \beta)\otimes X\\
&=&N\boxtimes_B(\sum_{\beta\in G^{(i,\alpha)}}a_{\beta\otimes 1}^B). 
\end{eqnarray*}
It follows that 
\begin{eqnarray*}
\<X,X\> &=&\<a^B_{1\otimes X}, a^B_{1\otimes X}\>\\
&=&\<a^B_{1\otimes X}, N\boxtimes_Ba^B_{\alpha'\otimes 1}\>\\
&=&\<a^B_{1\otimes X}\boxtimes a^B_{\alpha\otimes 1}, N\>\\
&=&\<N\boxtimes_B(\sum_{\beta\in G^{(i,\alpha)}}a_{\beta\otimes 1}^B), N\>\\
&=&\<o(G^{(i,\alpha)})N, N\>\\
&=& o(G^{(i,\alpha)}).
\end{eqnarray*}
On the other hand, by Lemma \ref{kwl2} and Proposition \ref{p3.6} we know that $\<(i,\alpha), (i,\alpha\>)=o(G^{(i,\alpha)}).$ This implies that  $(i,\alpha)=X.$ So $\alpha\otimes (i,\alpha)$ is a simple 
$B$-module.

Let $(i,\alpha)=\sum_{s=1}^tn_s X_s$ in $\CC_2$ where  $\{X_1,...,X_t\}$ are inequivalent simple objects and
$n_s>0.$ For fixed $s$ consider $a^B_{\alpha\otimes X_s}=\oplus_{\beta\in G^{(i,\alpha)}}\alpha\otimes ((1,\beta)\boxtimes X_s).$ Recall $N=\alpha\otimes (i,\alpha).$ Then 
$$\<N,a^B_{\alpha\otimes X_s}\>=\<N, \alpha\otimes X_s\>=\<(i,\alpha),X_s\>=n_s.$$
So $N$ is a $B$-submodule of $a^B_{\alpha\otimes X_s}.$ This implies that for any $1\leq s_1\leq t$ there exists $\beta_{s_1}\in G^{(i,\alpha)}$ such that $X_{s_1}=(1,\beta_{s_1})\boxtimes X_s.$ In particular, all
$X_s$ have the same dimension. 


(2) If $W^{\alpha}, W^{\gamma}$ are in the same orbit, then there exists $\beta\in G^i$ such that $\gamma\cong \beta\boxtimes \alpha.$ Since 
$$\sum_{\delta\in J_i}\delta\otimes (i,\delta)=M^i=a_{\beta\otimes 1}\boxtimes_AM^i=(\beta\otimes 1)\boxtimes (\sum_{\delta\in J_i} \delta\otimes (i,\delta))=\sum_{\delta\in J_i}(\beta\otimes \delta)\otimes (i,\delta)$$
we see that $(\beta\boxtimes \alpha)\otimes (i,\alpha)=\gamma\otimes (i,\alpha)$ is a direct summand of $\gamma\otimes (i,\gamma).$ This implies that $(i,\alpha)$ is a direct summand of $(i,\gamma).$ Similarly,
 $(i,\gamma)$ is a direct summand of $(i,\alpha).$ Thus  $(i,\alpha)$ and $(i,\gamma)$ are the same. 

 (3) From Proposition \ref{p3.6} we have
$$\<(i,\alpha), (i,\gamma)\>=\<a_{1\otimes (i,\alpha)}, a_{1\otimes (i,\gamma)}\>=\<M^i\boxtimes a_{\alpha'\otimes 1}, M^i\boxtimes a_{\gamma'\otimes 1}\>=\sum_{\epsilon\in \KW}N_{\alpha, \gamma'}^{\epsilon} N_{i',i}^k$$
where $k=a_{\epsilon\otimes 1}.$ Note that $N_{i',i}^k\ne 0$ if and only if $\epsilon\in G^i.$ Since $\gamma$ and 
$\alpha$ are not in the same $G^i$-orbit, it is immediately that $N_{\alpha, \gamma'}^{\epsilon}=0.$ 
This implies that $\<(i,\alpha), (i,\gamma)\>=0.$
\end{proof}			

Unfortunately, we can not determine the explicit decomposition of $(i,\alpha)$ into a direct sum of simple objects in $\CC_2$ for the categorical setting. But if $\CC_1,\CC_2, \CC$ are module categories of rational vertex operator algebras, the decomposition is given in the next section.

\section{Coset construction for vertex operator algebras}

In the rest of this paper, we assume that $U$ is a subalgebra of vertex operator algebra $V$ such that 

(1) $U=U^{cc},$ 

(2) $U, U^c, V$ are rational, $C_2$-cofinite and of stong CFT types,

(3) The conformal weights of any irredcuble $U,U^c,V$-modules are positive except $U,U^c,V.$  

Let $M^i$ for $\in I=\{1,...,p\}$ be the inequivalent irreducible $V$-modules with $M^1=V,$ 
$W^{\alpha}$ for $\alpha\in J=\{1,...,q\}$ be the inequivalent irreducible $U$-modules with $W^1=U,$ and $N^{\phi}$ for $\phi\in K=\{1,...,s\}$ be the inequivalent irreducible $U^c$-modules
with $N^1=U^c.$ Let $\CC_1=\CC_U$ be the $U$-module category, $\CC_2=\CC_{U^c}$ and 
$\CC=\CC_V=(\CC_1\otimes \CC_2)_V^0.$ So the results on categorical coset constructions are valid in the setting of vertex operator algebra. We will not repeat these results in this section.

Recall $G^i, G^{(i,\alpha)}, B$ from Section 8. Set $t=o(G^{(i,\alpha)}).$  Then we have
\begin{thm}
	Assume that $\Or(\CC_1^{kw})$ is a group and $G^{(i,\alpha)}$ is a cyclic subgroup. Then 
	$(i,\alpha)=\oplus_{s=1}^tX_s$ is a direct sum of inequivalent irreducible $U^c$-modules such that
	$\dim (X_s)=\frac{1}{t}d_{(i,\alpha)}=\frac{1}{t}d_id_{\alpha}$ for all $s.$
\end{thm}
\begin{proof} 
	Note that $B=\oplus_{\beta\in G^{(i,\alpha)}}B^{\beta}$  is a simple current extension of $U\otimes U^c$ where $B^\beta=\beta\otimes (1,\beta).$ Then  the dual group $D$ of $G^{(i,\alpha)}$ is an automorphism group of $B$ such that for any  $\chi\in D$ and $\beta\in G^{(i,\alpha)},$ $\chi$ acts on $B^{\beta}$ as $\chi(\beta).$ It is clear that $B^D=U\otimes U^c.$ Recall from the proof of Theorem \ref{KWd1} that $N=\alpha\otimes (i,\alpha)$ is an irreducible $B$-module. Let $D_N=\{\chi\in D| N\circ \chi\cong N\}$ be the stabilizer of 
$N$  \cite{DLM4}. It follows from \cite{DLM4} that $N$ is a projective module for $D_N.$ 
Since $D$ is cyclic, we see that $D_N$ is cyclic and $N$ is a module for $D_N.$ Moreover, $N$ is direct sum of inequivalent  irreducible $U\otimes U^c$-modules \cite{DRX}. The result follows immediately now. 
	\end{proof}
The above theorem is an analog of Lemma 2.1 of \cite{X2}.

\end{document}